\documentclass[english]{article}
\usepackage[T1]{fontenc}
\usepackage[latin9]{inputenc}
\usepackage{float}
\usepackage{amsthm}
\usepackage{amsmath}
\usepackage{amssymb}
\usepackage{graphicx}
\usepackage{esint}

\makeatletter
  \theoremstyle{remark}
  \newtheorem*{rem*}{\protect\remarkname}
  \theoremstyle{remark}
  \newtheorem*{acknowledgement*}{\protect\acknowledgementname}

\usepackage{amsthm}
\usepackage{indentfirst}
\usepackage{subfig}
\usepackage{color}
\@ifundefined{definecolor}

\theoremstyle{definition}
\newtheorem{theorem}{Theorem}[section]
\newtheorem{defn}[theorem]{Definition}
\newtheorem{lemma}[theorem]{Lemma}

\makeatother

\usepackage{babel}
  \providecommand{\acknowledgementname}{Acknowledgement}
  \providecommand{\remarkname}{Remark}

\begin{document}

\title{Synchrosqueezing-based Recovery of Instantaneous Frequency from Nonuniform
Samples}

\author{Gaurav Thakur%
\thanks{Program in Applied and Computational Mathematics, Princeton University,
Princeton, NJ 08544, USA%
} \ and Hau-Tieng Wu%
\thanks{Mathematics Department, Princeton University, Princeton, NJ 08544,
USA%
}}

\date{June 24, 2011}
\maketitle
\begin{abstract}
We propose a new approach for studying the notion of the instantaneous
frequency of a signal. We build on ideas from the Synchrosqueezing
theory of Daubechies, Lu and Wu and consider a variant of Synchrosqueezing,
based on the short-time Fourier transform, to precisely define the
instantaneous frequencies of a multi-component AM-FM signal. We describe
an algorithm to recover these instantaneous frequencies from the uniform
or nonuniform samples of the signal and show that our method is robust
to noise. We also consider an alternative approach based on the conventional,
Hilbert transform-based notion of instantaneous frequency to compare
to our new method. We use these methods on several test cases and
apply our results to a signal analysis problem in electrocardiography.
\end{abstract}


\section{\label{SecIntro}Introduction}

In a recent paper \cite{DLW10}, Daubechies, Lu and Wu proposed and
analyzed an adaptive wavelet-based signal analysis method that they
called {}``Synchrosqueezing.'' The authors considered continuous-domain
signals that are a superposition of a finite number of approximately
harmonic components, and they showed that Synchrosqueezing is able
to decompose an arbitrary signal of this type. They showed that for
such signals, Synchrosqueezing provides accurate instantaneous frequency
information about their constituent components.\\

The concept of instantaneous frequency is a natural extension of the
usual Fourier frequency that describes how fast a signal oscillates
locally at a given point in time, or more generally, the different
rates of oscillation at a given time. However, instantaneous frequency
has so far remained a somewhat heuristic concept and has lacked a
definition that is both mathematically rigorous and entirely satisfactory
\cite{HW09}. The analysis of signals from samples spaced nonuniformly
in time is also an important problem in several applications, arising
in radar detection, audio processing, seismology and many other fields
\cite{AG01,Ma01}. In this paper, we propose a new approach based
on Synchrosqueezing to precisely define the instantaneous frequencies
of a signal, and to recover these instantaneous frequencies from the
uniform or nonuniform samples of the signal.\\

We consider a class of multi-component AM-FM signals, comparable to
the type studied in \cite{DLW10}, and define a variant of Synchrosqueezing
based on ideas similar to those in \cite{DLW10} but using the short-time
Fourier transform (STFT). We show that by applying this modified Synchrosqueezing
transform to an impulse train weighted by the samples, we can determine
the instantaneous frequencies of the signal with high accuracy. We
show furthermore that this procedure is robust with respect to noise.
For purposes of comparison, we also consider a parallel approach based
on the conventional, Hilbert transform-based notion of instantaneous
frequency, combined with a well known least-squares method for bandlimited
signal reconstruction from nonuniform samples \cite{Gr99}. We apply
both methods to several test cases and compare their performance.
We also consider a problem concerning the extraction of respiration
data from a single-lead electrocardiogram (ECG) signal, and show how
these methods can be used to study it.\\

This paper is organized as follows. In Section \ref{SecBG}, we briefly
review the conventional approach to instantaneous frequency. In Section
\ref{SecMain}, we describe our Synchrosqueezing-based algorithm and
present the main new concepts and results of the paper. Our adaptation
of the least-squares method for bandlimited functions is described
in Section \ref{SecBL}. We then perform our numerical experiments
in Section \ref{SecNumerical} and discuss the application to ECG
analysis.

\section{\label{SecBG}Background Material}

We first discuss the precise meaning of the instantaneous frequency
(IF) of a signal and some of the obstacles encountered in computing
it. We start by making a few definitions. Let $f$ be a tempered distribution.
We denote the forward and inverse Fourier transforms of $f$ by $\hat{f}$
and $\check{f}$, using the normalization $\widehat{e^{-\pi x^{2}}}=e^{-\pi\xi^{2}}$.
For a fixed window function $g$ in the Schwartz class $\mathcal{S}$,
we denote the \textit{modified short-time Fourier transform} of $f$
by

\begin{equation}
V_{g}f(t,\eta):=\int_{-\infty}^{\infty}f(x)g(x-t)e^{-2\pi i\eta(x-t)}dx.\label{STFT}
\end{equation}
This is simply the regular STFT with a modulation factor $e^{2\pi i\eta t}$,
which will be convenient for our purposes. We will also occasionally
write $f_{1}\eqsim f_{2}$ if the inequality $C_{1}f_{1}\leq f_{2}\leq C_{2}f_{1}$
holds, where $C_{1}$ and $C_{2}$ are constants independent of $f_{1}$
and $f_{2}$. Finally, we will let $\mathrm{sinc}(x):=\frac{\sin(\pi x)}{\pi x}$.\\
 \\
 We now consider a function $f$ having the AM-FM form $f(t)=A(t)\cos(2\pi\phi(t))$.
We want to determine $\phi'(t)$, which intuitively describes the
local rate of oscillation of $f$ at $t$. This leads to the following
general definition.

\begin{defn}Suppose $f$ is a superposition of $K$ AM-FM components,
having the form

\begin{equation}
f(t)=\sum_{k=1}^{K}A_{k}(t)\cos(2\pi\phi_{k}(t))\label{IIF}
\end{equation}

with $\phi_{k}'(t)>0$ for all $k$. Then the \textit{ideal instantaneous
frequencies} $\mathrm{IIF}(f,A_{k},\phi_{k})$ are defined to be the
set of functions $\{\phi_{k}'(t)\}_{1\leq k\leq K}$. 

\end{defn}

\noindent Note that this concept only makes sense if all the component
functions $A_{k}$ and $\phi_{k}$ are known, and it is impossible
to define the IF of an arbitrary function $f$ in this way. In fact,
an arbitrary $f$ will generally not have a unique representation
of the form (\ref{IIF}), with many different choices of $A_{k}$,
$\phi_{k}$ and $K$ possible. Even if we restrict $K=1$, there is
in general no way to separate the amplitude factor $A_{1}$ from the
frequency factor $\cos(2\pi\phi_{1})$ without having some additional
information on $f$.\\

\noindent In data analysis applications, we typically only know the
full signal $f$ and we want to obtain an approximation to the IIF.
One of the most commonly used approaches for doing this can be described
as follows \cite{HW09}. For an appropriate $f$, we define the operator
$P^{+}f=\left(f+i\mathcal{H}f\right)/2$, where $\mathcal{H}f=(-i\mathrm{sign}(\eta)\hat{f}(\eta))\thinspace\check{}$
is the Hilbert transform of $f$. $P^{+}$ is known as the \textit{Riesz
projection} in mathematics and as the \textit{single-sideband modulation}
or the \textit{analytic signal} in the engineering literature. We
then consider the \textit{Hilbert transform IF} given by 
\begin{equation}
\mathrm{IF}_{H}f(t)=\frac{1}{2\pi}\mathrm{Im}\left(\frac{\partial_{t}P^{+}f(t)}{P^{+}f(t)}\right).\label{DefIF}
\end{equation}

The motivation for this concept is that the function $f$ is assumed
to have the form (\ref{IIF}) with a single AM-FM component, $f(t)=A(t)\cos(2\pi\phi(t))$,
and under some conditions, $\mathrm{IF}_{H}f(t)$ is a good approximation
to $\phi'(t)$. The \textit{Bedrosian theorem} states that if $\mathrm{supp}(\hat{A})$
and $\mathrm{supp}(\widehat{\cos(2\pi\phi)})$ are disjoint, then
$P^{+}f(t)=A(t)\cdot P^{+}\cos(2\pi\phi(t))$. If we additionally
assume that $\mathrm{supp}(\widehat{\exp(2\pi i\phi)})\subset[0,\infty)$,
then $P^{+}\cos(2\pi\phi(t))=e^{2\pi i\phi(t)}$, and we have $\mathrm{IF}_{H}f(t)=\frac{1}{2\pi}\mathrm{Im}\left(2\pi i\phi'(t)+\frac{A'(t)}{A(t)}\right)=\phi'(t)$.\\
 \\
 These conditions on $A$ and $\phi$ are fairly restrictive and can
be hard to verify for real-world signals, particularly the requirement
that $\mathrm{supp}(\widehat{\exp(2\pi i\phi)})\subset[0,\infty)$.
Even when they hold, the computation of $\mathrm{IF}_{H}$ is often
sensitive to noise and numerical roundoff errors due to the Hilbert
transform computation and the possible cancellation of zeros in the
numerator and denominator of (\ref{DefIF}). In spite of this, $\mathrm{IF}_{H}$
has proven to give meaningful results for signals arising from a variety
of applications, and is in widespread use in data analysis. We refer
to the papers \cite{HW09} and \cite{BP97} for more details on the
physical interpretation of $\mathrm{IF}_{H}$. In Section \ref{SecMain},
we will propose an alternative approach based on different ideas to
approximate the IIF of a signal. We return to $\mathrm{IF}_{H}$ in
Section \ref{SecBL} and show how it can be computed from nonuniform
samples of a bandlimited signal.

\section{\label{SecMain}Synchrosqueezing with the Short-Time Fourier Transform}

Synchrosqueezing is an approach originally introduced in the context
of audio signal analysis in \cite{DM96} and was recently developed
further in \cite{DLW10}. Synchrosqueezing belongs to the family of
time-frequency reassignment methods and is a nonlinear operator that
{}``sharpens'' the time-frequency plot of a signal's continuous
wavelet transform so that it provides more useful information about
the IF. In contrast to the reassignment methods discussed in \cite{ACMF03},
Synchrosqueezing is highly adaptive to the given signal and largely
independent of the particular wavelet used, and also allows the signal
to be reconstructed from the reassigned wavelet coefficients. We refer
to the paper \cite{DLW10} for more details.\\

In this paper, we take a slightly different approach to Synchrosqueezing
than in \cite{DLW10}, based on the modified short-time Fourier transform.
We will develop our theory independently of the results in \cite{DLW10}
and show that this new Synchrosqueezing transform provides a way to
estimate the IIF of a given function from discrete, nonuniform samples
of the function. We will make a series of definitions leading up to
our main theorem. We first define the following class of functions.

\begin{defn}\textbf{(Intrinsic Mode Functions})\\
 The space $\mathcal{B}_{\epsilon}$ of \textit{Intrinsic Mode Functions
(IMFs) of type B} consists of functions $f$ having the form 
\begin{equation}
f(t)=A(t)e^{2\pi i\phi(t)}\label{IMF}
\end{equation}
 such that for some fixed $\epsilon\ll1$, $A$ and $\phi$ satisfy
the following conditions:

\[
A,\phi\in L^{\infty}\cap C^{\infty}\mbox{,}~A^{(m)},\phi^{(m)}\in L^{\infty}\,\mathrm{for\, all\,}m,~A(t)>0,~\phi'(t)>0,
\]

\[
\left\Vert A'\right\Vert _{L^{\infty}}\leq\epsilon\left\Vert \phi'\right\Vert _{L^{\infty}},\quad\left\Vert \phi''\right\Vert _{L^{\infty}}\leq\epsilon\left\Vert \phi'\right\Vert _{L^{\infty}}.
\]

\end{defn}

\noindent We then consider the function class $\mathcal{B}_{\epsilon,d}$,
defined as follows.

\begin{defn}\textbf{\label{DefBClass}(Superpositions of IMFs})\\
 The space $\mathcal{B}_{\epsilon,d}$ of \textit{superpositions of
IMFs} consists of functions $f$ having the form 
\[
f(t)=\sum_{k=1}^{K}f_{k}(t)
\]
 for some $K>0$ and $f_{k}=A_{k}e^{2\pi i\phi_{k}}\in\mathcal{B}_{\epsilon}$
such that the $\phi_{k}$ satisfy 
\[
\inf_{t}\phi'_{k}(t)-\sup_{t}\phi'_{k-1}(t)>d.
\]

\end{defn}

Our main results in this paper will be stated for $\mathcal{B}_{\epsilon,d}$
functions. Intuitively, functions in $\mathcal{B}_{\epsilon,d}$ are
composed of several oscillatory components with slowly time-varying
amplitudes and IIFs (we call these components IMFs, following \cite{DLW10}),
and the IIFs of any two consecutive components are separated by at
least $d$. This function class is fairly restrictive, but it turns
out to be a reasonable model of signals that arise in many applications,
and the conditions on the IMFs will allow us to obtain accurate estimates
of the IIF of $f$. Note, however, that $\mathcal{B}_{\epsilon,d}$
is not a vector space. We next define a closely related notion.

\begin{defn}\textbf{(Impulse Trains)}

Let $T>0$ and $\{a_{n}\}\in l^{\infty}$ with $\|\{a_{n}\}\|_{l^{\infty}}\leq T^{2}$.
The \textit{impulse train class} is defined by

\[
\mathcal{D}_{\epsilon,d}^{T,\{a_{n}\}}=\left\{ \sum_{n=-\infty}^{\infty}(T+a_{n+1}-a_{n})\delta(t-Tn-a_{n})f(t):~~f\in\mathcal{B}_{\epsilon,d}\right\} .
\]

\end{defn}

\noindent We can treat elements of the impulse train class $\mathcal{D}_{\epsilon,d}^{T,\{a_{n}\}}$
as tempered distributions. As $T\to0$, $\tilde{f}\in\mathcal{D}_{\epsilon,d}^{T,\{a_{n}\}}$
converges weakly to $f\in\mathcal{B}_{\epsilon,d}$ with respect to
Schwartz test functions, so $\mathcal{D}_{\epsilon,d}^{T,\{a_{n}\}}$
can be thought of as a sampled version of the continuous-domain $\mathcal{B}_{\epsilon,d}$
class. In applications, we are given a collection of nonuniformly
spaced samples of the form $\{f(t_{n})\}$, $t_{n}=Tn+a_{n}$, where
$\{a_{n}\}$ represents a small perturbation from uniform samples
$\{Tn\}$ with sampling interval $T$. This can be converted to an
element in the class $\mathcal{D}_{\epsilon,d}^{T,\{a_{n}\}}$ by
multiplying $f(t_{n})$ by a factor $t_{n+1}-t_{n}=T+a_{n+1}-a_{n}$.\\

For a given window function $g\in\mathcal{S}$, we can apply the modified
short-time Fourier transform (\ref{STFT}) to any $\tilde{f}\in\mathcal{D}_{\epsilon,d}^{T,\{a_{n}\}}$,
with the tempered distribution $\tilde{f}$ acting on $g(\cdot-t)e^{-2\pi i\eta(\cdot-t)}\in\mathcal{S}$.
When $|V_{g}\tilde{f}(t,\eta)|>0$, we define the \textit{instantaneous
frequency information $\omega\tilde{f}(t,\eta)$} by 
\begin{equation}
\omega\tilde{f}(t,\eta)=\frac{\partial_{t}V_{g}\tilde{f}(t,\eta)}{2\pi iV_{g}\tilde{f}(t,\eta)}.\label{IFI}
\end{equation}

\noindent The idea with (\ref{IFI}) is that $\partial_{t}V_{g}\tilde{f}$
is a first approximation to the IF of $f$, and dividing by $V_{g}\tilde{f}$
{}``removes the influence'' of the window $g$ and sharpens the
resulting time-frequency plot. We can use this to consider the following
operator, which will be our main computational tool in this paper.

\begin{defn}\textbf{(STFT Synchrosqueezing})\\
 For $\tilde{f}\in\mathcal{D}_{\epsilon,d}^{T,\{a_{n}\}}$, the \textit{STFT
Synchrosqueezing transform with resolution $\alpha>0$ and threshold
$\gamma\geq0$ } is defined by 
\begin{equation}
S^{\alpha,\gamma}\tilde{f}(t,\xi)=\left|\left\{ \eta:~|\xi-\omega\tilde{f}(t,\eta)|<\frac{\alpha}{2},~|V_{g}\tilde{f}(t,\eta)|\geq\gamma,~0\leq\eta\leq\frac{1}{T}\right\} \right|\label{STFTSS}
\end{equation}
 where $(t,\xi)\in\mathbb{R}\times\alpha\mathbb{N}$ and $\left|\cdot\right|$
denotes the Lebesgue measure on $\mathbb{R}$.

\end{defn}

\noindent $S^{\alpha,\gamma}\tilde{f}$ is a kind of highly concentrated
version of $\omega\tilde{f}$ that {}``squeezes'' the content of
$\omega\tilde{f}$ closer to the IF curves in the time-frequency plane.
We are finally in a position to define an alternative notion of instantaneous
frequency, based on the concepts discussed above. 

\noindent \begin{defn}\textbf{(Synchrosqueezing-based Instantaneous
Frequency})\\
 \textit{The Synchrosqueezing-based IF with resolution $\alpha$ and
threshold $\gamma$} of $f\in\mathcal{B}_{\epsilon,d}$, estimated
from a corresponding $\tilde{f}\in\mathcal{D}_{\epsilon,d}^{T,\{a_{n}\}}$,
is a set-valued function $\mathrm{IF}_{S}f(t):\mathbb{R}\rightarrow2^{\alpha\mathbb{N}}$
given by
\[
\mathrm{IF}_{S}f(t)=\left\{ \alpha n:n\in\mathbb{N},S^{\alpha,\gamma}\tilde{f}(t,\alpha n)>0\right\} .
\]

\noindent \end{defn}

As a motivating example of this concept, let $f(t)=e^{2\pi it}\in\mathcal{B}_{\epsilon,d}$.
The IIF is clearly the singleton $\{1\}$. For the window $g(t)=e^{-\pi t^{2}}$,
we can compute $V_{g}f(t,\eta)=e^{2\pi it-\pi(\eta-1)^{2}}$, so $\omega(t,\eta)=1$
for all $(t,\eta)$. This means that for $\gamma=0$, $S^{\alpha,\gamma}f(t,\xi)$
will be supported on $\{\lfloor\frac{1}{\alpha}\rfloor\alpha,\lceil\frac{1}{\alpha}\rceil\alpha\}$
for all $t$, which is close to $\{1\}$ for small $\alpha$. We will
show in Theorem \ref{ThmMain} that if $\tilde{f}\in\mathcal{D}_{\epsilon,d}^{T,\{a_{n}\}}$
is a sampled approximation of $f$ and $\alpha$ and $\gamma$ are
chosen appropriately, then it has the same $\mathrm{IF}_{S}$ set.\\

We will use several auxiliary notations in the statement and proof
of our main theorem. Let
\begin{equation}
\tilde{f}(t)=\sum_{k=1}^{K}\sum_{n=-\infty}^{\infty}\delta(t-Tn-a_{n})(T+a_{n+1}-a_{n})A_{k}(t)e^{2\pi i\phi_{k}(t)}\in\mathcal{D}_{\epsilon,d}^{T,\{a_{n}\}}.\label{ftilde}
\end{equation}
We then define the following expressions:
\begin{eqnarray*}
I_{n}^{(m)} & := & \int_{-\infty}^{\infty}|u|^{n}|g^{(m)}(u)|du\\
\mathcal{I} & := & \bigcup_{n\in\mathbb{Z}}\mathcal{I}_{n},\quad\mathcal{I}_{n}=[Tn,Tn+a_{n}]\mathrm{\, if\,}a_{n}\geq0\,\mathrm{or\,}[Tn+a_{n},Tn]\mathrm{\, if\,}a_{n}<0\\
E_{1} & := & \sum_{k=1}^{K}\epsilon\left\Vert \phi_{k}'\right\Vert _{L^{\infty}}\left(I_{1}+\pi\left\Vert A_{k}\right\Vert _{L^{\infty}}I_{2}\right)\\
E_{1}' & := & \sum_{k=1}^{K}\epsilon\left\Vert \phi_{k}'\right\Vert _{L^{\infty}}\left(\frac{1}{2\pi}I_{0}+\left(\left\Vert A_{k}\right\Vert _{L^{\infty}}+\left\Vert \phi_{k}'\right\Vert _{L^{\infty}}\right)I_{1}+\pi\left\Vert A_{k}\right\Vert _{L^{\infty}}\left\Vert \phi_{k}'\right\Vert _{L^{\infty}}I_{2}\right)\\
E_{3} & := & \sup_{1\leq k\leq K}\left\Vert \phi_{k}'\right\Vert _{L^{\infty}}
\end{eqnarray*}

\noindent This leads to the following results.

\begin{theorem}\label{ThmMain} Let $0\leq T\leq1$. Suppose we have
$\tilde{f}\in\mathcal{D}_{\epsilon,d}^{T,\{a_{n}\}}$ as in (\ref{ftilde})
with $\left\Vert \{a_{n}\}\right\Vert _{l^{\infty}}\leq T^{2}$, and
a window function $g\in\mathcal{S}$ with $\mbox{supp}(\hat{g})\subset[-\frac{d}{2},\frac{d}{2}]$
and $\int_{\mathcal{I}}|g(x)|+|g'(x)|dx\leq\frac{\kappa}{T}\left\Vert \{a_{n}\}\right\Vert _{l^{\infty}}$
for a constant $\kappa$. Then there exist numbers $E_{2}=E_{2}(A_{k},\phi_{k},g)$
and $E_{2}'=E_{2}(A_{k},\phi_{k},g)$ such that if we have a given
resolution $\alpha$ satisfying
\[
\alpha\geq\frac{2(E_{1}'+E_{2}')}{E_{1}+E_{2}}+2E_{3},
\]
then the following statements hold.
\begin{enumerate}
\item Let $0\leq\eta\leq\frac{1}{T}$ and fix $k$, $1\leq k\leq K$. For
each pair $(t,\eta)\in Z_{k}:=\{(t,\eta):|\eta-\phi_{k}'(t)|<\frac{d}{2}\}$
with $|V_{g}\tilde{f}(t,\eta)|>E_{1}+E_{2}$, we have $|\omega\tilde{f}(t,\eta)-\phi_{k}'(t)|<\frac{\alpha}{2}.$
If $(t,\eta)\not\in Z_{k}$ for any $k$, then $|V_{g}\tilde{f}(t,\eta)|\leq E_{1}+E_{2}$.
\item Suppose we have a threshold $\gamma$ such that $E_{1}+E_{2}<\gamma\leq|V_{g}\tilde{f}(t,\eta)|$
for all $(t,\eta)\in Z_{k}$. Then for all $t$, $S^{\alpha,\gamma}\tilde{f}(t,\xi)$
is supported in the $2K$-point set $\bigcup_{1\leq k\leq K}\{\lfloor\frac{\phi_{k}'(t)}{\alpha}\rfloor\alpha,\lceil\frac{\phi_{k}'(t)}{\alpha}\rceil\alpha\}$.
\end{enumerate}
\end{theorem}

Theorem \ref{ThmMain} says that the $\mathrm{IF}_{S}$ defined by
STFT Synchrosqueezing approximates the IIF set accurately up to the
preassigned resolution $\alpha$, without knowing anything about the
symbolic form of $f$. The result also does not depend on the precise
shape of the window $g$ that we use, so the $\mathrm{IF}_{S}$ is
in a sense adaptive to the structure of the sampled function $\tilde{f}$.
The procedure suggested by Theorem \ref{ThmMain} can be implemented
as follows. We first discretize $\tilde{f}$ on a uniform grid (finer
than $T$) by zero-padding in between the impulses. For each $t$,
we can then compute $V_{g}\tilde{f}$ and $\partial_{t}V_{g}\tilde{f}=-V_{g'}\tilde{f}+2\pi i\eta V_{g}\tilde{f}$
on using FFTs and use the results to approximate $S^{\alpha,\gamma}\tilde{f}(t,\alpha n)$
for $n\in\mathbb{N}$. In practice, the upper bound on $\eta$ in
(\ref{STFTSS}) is determined by how finely $\tilde{f}$ is discretized,
and as long as the threshold $\gamma$ is not too small, we ignore
the locations where the denominator in (\ref{IFI}) is close to zero
and avoid numerical stability issues. We finally find the (numerical)
support of $S^{\alpha,\gamma}\tilde{f}(t,\cdot)$ to determine the
component(s) of $\mathrm{IF}_{S}f(t)$.\\

There is a tradeoff between $\alpha$ and $\gamma$ and the fluctuation
of the IIF components, and this is a kind of uncertainty principle
that is inherent to the $\mathrm{IF}_{S}$ concept. The $\mathrm{IF}_{S}$
is only meaningful up to the resolution $\alpha$, and beyond that,
we cannot approximate the IIF to any further level of accuracy. In
the simplest case when the amplitudes $\{A_{k}\}$ and IIF components
$\{\phi_{k}'\}$ are all constant, the lower bound on $\alpha$ will
be small and will allow us to specify a fine resolution $\alpha$,
resulting in a very accurate estimate of the IIF. On the other hand,
if the $\{A_{k}'\}$ or $\{\phi_{k}''\}$ are large, then the lower
bound on $\gamma$ will be significant, making $S^{\alpha,\gamma}\tilde{f}(t,\xi)=0$
for most $t$. In physical terms, the existence of a $\gamma$ in
Theorem \ref{ThmMain} ensures that the magnitude of each component
is not too small and its IIF is not too rapidly oscillating, or else
the component would be indistinguishable from noise.\\

The proof of Theorem \ref{ThmMain} involves a series of estimates.
Let $f$ and $\tilde{f}$ be given as in Theorem \ref{ThmMain}. In
what follows, we let $Q_{k}(t,\eta)=A_{k}(t)e^{2\pi i\phi_{k}(t)}\hat{g}(\eta-\phi_{k}'(t))$
to simplify some notation.

\begin{lemma}\label{ThmMainP1} If $(t,\eta)\in Z_{k}$ for some
fixed $k$, $1\leq k\leq K$, then
\begin{equation}
\left|V_{g}f(t,\eta)-Q_{k}(t,\eta)\right|\leq E_{1}\label{L1E1}
\end{equation}
and
\begin{equation}
\left|\frac{1}{2\pi i}\partial_{t}V_{g}f(t,\eta)-\phi_{k}'(t)Q_{k}(t,\eta)\right|\leq E_{1}'.\label{L1E2}
\end{equation}
If $(t,\eta)\not\in Z_{k}$ for any $k$, then
\[
\left|V_{g}f(t,\eta)\right|\leq E_{1}\mathrm{\quad and\quad}\left|\frac{1}{2\pi i}\partial_{t}V_{g}f(t,\eta)\right|\leq E_{1}'.
\]
\end{lemma}
\begin{proof}
Note that for any $l$, the bandwidth condition on $g$ shows that
$Q_{l}(t,\eta)=0$ whenever $(t,\eta)\not\in Z_{l}$. We assume $(t,\eta)\in Z_{k}$
for some $k$, as the other case can be done in exactly the same way.
For the first bound (\ref{L1E1}), Taylor expansions show that
\[
|1-e^{2\pi i(-\phi_{k}(x)+\phi_{k}(t)+(x-t)\phi_{k}'(t))}|\leq\pi\left\Vert \phi_{k}''\right\Vert _{L^{\infty}}|x-t|^{2}.
\]

\noindent It follows that {\allowdisplaybreaks
\begin{eqnarray*}
 &  & \left|V_{g}f(t,\eta)-Q_{k}(t,\eta)\right|\\
 & = & \left|V_{g}f(t,\eta)-\sum_{k=1}^{K}Q_{k}(t,\eta)\right|\\
 & \leq & \sum_{k=1}^{K}\bigg[\int_{-\infty}^{\infty}\left|A_{k}(x)-A_{k}(t)\right|\left|g(x-t)e^{-2\pi i\eta(x-t)}\right|dx\\
 &  & \quad+|A_{k}(t)|\int_{-\infty}^{\infty}\left|e^{2\pi i\phi_{k}(x)}-e^{2\pi i(\phi_{k}(t)+(x-t)\phi_{k}'(t))}\right|\left|g(x-t)e^{-2\pi i\eta(x-t)}\right|dx\bigg]\\
 & \leq & \sum_{k=1}^{K}\left(\int_{-\infty}^{\infty}\left\Vert A_{k}'\right\Vert _{L^{\infty}}|x-t|\left|g(x-t)\right|dx+|A_{k}(t)|\int_{-\infty}^{\infty}\pi\left\Vert \phi_{k}''\right\Vert _{L^{\infty}}|x-t|^{2}\left|g(x-t)\right|dx\right)\\
 & \leq & \sum_{k=1}^{K}\epsilon\left\Vert \phi_{k}'\right\Vert _{L^{\infty}}\left(I_{1}+\pi\left\Vert A_{k}\right\Vert _{L^{\infty}}I_{2}\right).
\end{eqnarray*}

\noindent We follow similar arguments for the second bound (\ref{L1E2}).
We first have the estimate
\begin{eqnarray*}
 &  & |A_{k}(x)\phi_{k}'(x)e^{2\pi i\phi_{k}(x)}-A_{k}(t)\phi_{k}'(t)e^{2\pi i(\phi_{k}(t)+\phi_{k}'(t)(x-t))}|\\
 & \leq & |A_{k}(x)\phi_{k}'(x)-A_{k}(t)\phi_{k}'(t)|+|1-e^{2\pi i(-\phi_{k}(x)+\phi_{k}(t)+\phi_{k}'(t)(x-t))}||A_{k}(t)\phi_{k}'(t)|\\
 & \leq & \left(\left\Vert A_{k}\right\Vert _{L^{\infty}}\left\Vert \phi_{k}''\right\Vert _{L^{\infty}}+\left\Vert A_{k}'\right\Vert _{L^{\infty}}\left\Vert \phi_{k}'\right\Vert _{L^{\infty}}\right)|x-t|+\pi\left\Vert A_{k}\right\Vert _{L^{\infty}}\left\Vert \phi_{k}'\right\Vert _{L^{\infty}}\left\Vert \phi_{k}''\right\Vert _{L^{\infty}}|x-t|^{2}.
\end{eqnarray*}
Thus we obtain
\begin{eqnarray*}
 &  & \left|\frac{1}{2\pi i}\partial_{t}V_{g}f(t,\eta)-\phi_{k}'(t)Q_{k}(t,\eta)\right|\\
 & = & \bigg|\sum_{k=1}^{K}\bigg[-\frac{1}{2\pi i}\int_{-\infty}^{\infty}A_{k}(x)e^{2\pi i\phi_{k}(x)}\partial_{x}(g(x-t)e^{-2\pi i\eta(x-t)})dx\\
 &  & \quad-\phi_{k}'(t)A_{k}(t)\int_{-\infty}^{\infty}e^{2\pi i(\phi_{k}(t)+\phi_{k}'(t)(x-t))}g(x-t)e^{-2\pi i\eta(x-t)}dx\bigg]\bigg|\\
 & \leq & \sum_{k=1}^{K}\bigg[\int_{-\infty}^{\infty}\left|A_{k}(x)\phi_{k}'(x)e^{2\pi i\phi_{k}(x)}-A_{k}(t)\phi_{k}'(t)e^{2\pi i(\phi_{k}(t)+\phi_{k}'(t)(x-t))}\right|\left|g(x-t)\right|dx\\
 &  & \quad+\frac{1}{2\pi}\int_{-\infty}^{\infty}|A_{k}'(x)||g(x-t)|dx\bigg]\\
 & \leq & \sum_{k=1}^{K}\epsilon\left\Vert \phi_{k}'\right\Vert _{L^{\infty}}\left(\left\Vert A_{k}\right\Vert _{L^{\infty}}I_{1}+\left\Vert \phi_{k}'\right\Vert _{L^{\infty}}I_{1}+\pi\left\Vert A_{k}\right\Vert _{L^{\infty}}\left\Vert \phi_{k}'\right\Vert _{L^{\infty}}I_{2}+\frac{1}{2\pi}I_{0}\right).
\end{eqnarray*}
}
\end{proof}
\noindent \begin{lemma}\label{ThmMainP2}Let $0\leq\eta\leq\frac{1}{T}$.
Then
\begin{equation}
\left|V_{g}\tilde{f}(t,\eta)-V_{g}f(t,\eta)\right|\leq E_{2}\label{L2E1}
\end{equation}
and
\begin{equation}
\frac{1}{2\pi}\left|\partial_{t}V_{g}\tilde{f}(t,\eta)-\partial_{t}V_{g}f(t,\eta)\right|\leq E_{2}'.\label{L2E2}
\end{equation}

\noindent for some numbers $E_{2}=E_{2}(A_{k},\phi_{k},g)$ and $E_{2}'=E_{2}'(A_{k},\phi_{k},g)$
independent of $\eta$ or $T$.

\noindent \end{lemma}
\begin{proof}
\noindent Suppose $h\in C^{1}$. Then denoting $t_{n}=Tn+a_{n}$,
we have
\begin{eqnarray}
 &  & \left|\sum_{n=-\infty}^{\infty}(t_{n+1}-t_{n})h(t_{n})-\int_{-\infty}^{\infty}h(t)dt\right|\nonumber \\
 & = & \left|\sum_{n=-\infty}^{\infty}\int_{t_{n}}^{t_{n+1}}h'(u)(u-t_{n+1})du\right|.\label{hBound}\\
 & = & \left|\int_{-\infty}^{\infty}h'(u)\sum_{n=-\infty}^{\infty}(u-t_{n+1})\chi_{[t_{n},t_{n+1}]}(u)du\right|.\nonumber 
\end{eqnarray}

\noindent Since $V_{g}\tilde{f}(t,\eta)=\sum_{n=-\infty}^{\infty}(t_{n+1}-t_{n})f(t_{n})g(t_{n}-t)e^{-2\pi i\eta(t_{n}-t)}$,
we use the above calculation to find that
\begin{eqnarray}
 &  & \left|V_{g}\tilde{f}(t,\eta)-V_{g}f(t,\eta)\right|\nonumber \\
 & \leq & \sum_{k=1}^{K}\left|\int_{-\infty}^{\infty}\partial_{u}\left(A_{k}(u)e^{2\pi i(\phi_{k}(u)-(u-t)\eta)}g(u-t)\right)\sum_{n=-\infty}^{\infty}(u-t_{n+1})\chi_{[t_{n},t_{n+1}]}(u)du\right|\nonumber \\
 & \leq & \sum_{k=1}^{K}\left(\left|\int_{-\infty}^{\infty}G(u)C(u)du\right|+2\epsilon T\left\Vert \phi_{k}'\right\Vert _{L^{\infty}}I_{0}+2T\left\Vert A_{k}\right\Vert _{L^{\infty}}I_{0}'\right),\label{GuCu}
\end{eqnarray}

\noindent where we let $G(u):=2\pi iA_{k}(u)(\phi_{k}'(u)-\eta)g(u-t)e^{2\pi i(\phi_{k}(u)-u\eta)}$
and $C(u):=\sum_{n=-\infty}^{\infty}(u-t_{n+1})\chi_{[t_{n},t_{n+1}]}(u)$.
For brevity, we will fix $k$ and omit the subscripts on $A_{k}$
and $\phi_{k}$ in what follows. The function $C(u)$ is well approximated
by the uniform sawtooth function
\begin{eqnarray*}
W(u) & = & \sum_{n=-\infty}^{\infty}(u-T(n+1))\chi_{[Tn,T(n+1)]}(u)=-\sum_{n=1}^{\infty}\frac{\sin(2\pi nu/T)}{\pi n/T}-\frac{T}{2},
\end{eqnarray*}

\noindent where the last equality holds for almost all $u$. We define
the difference $D(u)=W(u)-C(u)$. Since $G\in\mathcal{S}$, using
the Parseval theorem gives
\begin{eqnarray}
 &  & \left|\int_{-\infty}^{\infty}G(u)C(u)du\right|\nonumber \\
 & = & \left|-\int_{-\infty}^{\infty}G(u)\sum_{n=1}^{\infty}\frac{\sin(2\pi nu/T)}{\pi n/T}du-\hat{G}(0)\frac{T}{2}+\int_{-\infty}^{\infty}G(u)D(u)du\right|\nonumber \\
 & \leq & \left|\sum_{n=1}^{\infty}\frac{T}{4\pi in}\left(\hat{G}\left(\frac{2\pi n}{T}\right)-\hat{G}\left(-\frac{2\pi n}{T}\right)\right)\right|+\frac{T}{2}|\hat{G}(0)|+\left|\int_{-\infty}^{\infty}G(u)D(u)du\right|.\label{3Terms}
\end{eqnarray}

\noindent We now consider each term in (\ref{3Terms}) separately.
For the first term, we will need an estimate of $\hat{G}(\xi)$. We
claim that $\hat{G}$ will be insignificant outside intervals centered
at $\phi'(u)-\eta$. For any $J>0$, let $\xi$ be such that $|\phi'(u)-\eta-\xi|\geq J$.
Then we have
\begin{eqnarray}
|\hat{G}(\xi)| & = & 2\pi\left|\int_{-\infty}^{\infty}A(u)(\phi'(u)-\eta)e^{2\pi i(\phi(u)-(\eta+\xi)u)}g(u-t)du\right|\nonumber \\
 & = & 2\pi\left|\int_{-\infty}^{\infty}e^{2\pi i(\phi(u)-(\eta+\xi)u)}\partial_{u}\left(\frac{A(u)(\phi'(u)-\eta)g(u-t)}{\phi'(u)-\eta-\xi}\right)du\right|\nonumber \\
 & \leq & 2\pi\bigg(\frac{J(\epsilon\left\Vert A\right\Vert _{L^{\infty}}I_{0}+\epsilon\left\Vert \phi'\right\Vert _{L^{\infty}}I_{0}+\left\Vert \phi'\right\Vert _{L^{\infty}}\left\Vert A\right\Vert _{L^{\infty}}I_{0}')+\epsilon\left\Vert \phi'\right\Vert _{L^{\infty}}\left\Vert A\right\Vert _{L^{\infty}}I_{0}}{J^{2}}\nonumber \\
 &  & \quad+\eta\frac{J(\epsilon I_{0}+\left\Vert A\right\Vert _{L^{\infty}}I{}_{0}')+\epsilon\left\Vert A\right\Vert _{L^{\infty}}I_{0}}{J^{2}}\bigg)\nonumber \\
 & \leq & C_{1}(\eta+1)\left(\frac{1}{J}+\frac{1}{J^{2}}\right),\label{GHatBound1}
\end{eqnarray}

\noindent where the constant $C_{1}=C_{1}(A_{k},\phi_{k},g)$ is independent
of $\eta$ or $J$. We take $J=\max(\frac{2\pi n}{T}-L^{+},L^{-}-\frac{2\pi n}{T},\frac{\pi}{T})$
with the numbers $L^{+}$ and $L^{-}$ chosen such that $L^{-}\leq|\phi'(u)-\eta|\leq L^{+}$
for all $u$. This gives
\begin{eqnarray}
 &  & \left|\sum_{n=1}^{\infty}\frac{T}{4\pi in}\left(\hat{G}\left(-\frac{2\pi n}{T}\right)-\hat{G}\left(\frac{2\pi n}{T}\right)\right)\right|\nonumber \\
 & \leq & C_{1}(\eta+1)\sum_{n=1}^{\infty}\frac{2T}{4\pi n\max(\frac{2\pi n}{T}-L^{+},L^{-}-\frac{2\pi n}{T},\frac{\pi}{T})}\nonumber \\
 & \leq & C_{1}(\eta+1)\sum_{n=1}^{\infty}2\max(|L^{+}-L^{-}|,1)\left(\frac{T}{2\pi n}\right)^{2}\nonumber \\
 & = & C_{1}\max(|L^{+}-L^{-}|,1)\frac{T^{2}(\eta+1)}{12}.\label{Term1}
\end{eqnarray}

\noindent For the second term in (\ref{3Terms}), we consider two
cases. If $|\phi'(u)-\eta|\geq J$, then calculations similar to those
in (\ref{GHatBound1}) show that
\begin{eqnarray*}
|\hat{G}(0)| & = & 2\pi\left|\int_{-\infty}^{\infty}e^{2\pi i(\phi(u)-\eta u)}\partial_{u}\left(\frac{A'(u)g(u-t)+A(u)g'(u-t)}{\phi'(u)-\eta}\right)du\right|\\
 & \leq & 2\pi\left(\frac{\left\Vert A''\right\Vert _{L^{\infty}}I_{0}+2\left\Vert A'\right\Vert _{L^{\infty}}I_{0}'+\left\Vert A\right\Vert _{L^{\infty}}I_{0}''}{J}+\frac{\left\Vert \phi'\right\Vert _{L^{\infty}}\left\Vert A'\right\Vert _{L^{\infty}}I_{0}+\left\Vert \phi'\right\Vert _{L^{\infty}}\left\Vert A\right\Vert _{L^{\infty}}I_{0}'}{J^{2}}\right)\\
 & \leq & C_{2}\left(\frac{1}{J}+\frac{1}{J^{2}}\right)
\end{eqnarray*}

\noindent for some constant $C_{2}$. On the other hand, if $|\phi'(u)-\eta|<J$,
then we can easily estimate $|\hat{G}(0)|\leq\left\Vert G\right\Vert _{L^{1}}\leq2\pi\left\Vert A\right\Vert _{L^{\infty}}JI_{0}.$
We now choose a $J$ that minimizes $\max(C_{2}\left(\frac{1}{J}+\frac{1}{J^{2}}\right),2\pi\left\Vert A\right\Vert _{L^{\infty}}JI_{0})$
over all $0\leq\eta\leq\frac{1}{T}$. This results in a bound of the
form

\noindent 
\begin{equation}
\frac{T}{2}|\hat{G}(0)|\leq C_{3}T\label{Term2}
\end{equation}

\noindent for a constant $C_{3}$. The third term in (\ref{3Terms})
is controlled by the nonuniform perturbation $\{a_{n}\}$. Note that
if $u\not\in\mathcal{I}$, then $|D(u)|\leq2\left\Vert \{a_{n}\}\right\Vert _{l^{\infty}}$,
and for $u\in\mathcal{I}$, we have $|D(u)|\leq T+2\left\Vert \{a_{n}\}\right\Vert _{l^{\infty}}$.
This gives

\noindent 
\begin{eqnarray}
 &  & \left|\int_{-\infty}^{\infty}G(u)D(u)du\right|\nonumber \\
 & = & \left|\int_{\mathcal{I}}G(u)D(u)du+\int_{\mathbb{R}\backslash\mathcal{I}}G(u)D(u)du\right|\nonumber \\
 & \leq & \left\Vert \phi'-\eta\right\Vert _{L^{\infty}}\left\Vert A\right\Vert _{L^{\infty}}\left((T+\left\Vert \{a_{n}\}\right\Vert _{l^{\infty}})\int_{\mathcal{I}}|g(u)|du+I_{0}\left\Vert \{a_{n}\}\right\Vert _{l^{\infty}}\right)\nonumber \\
 & \leq & (\left\Vert \phi'\right\Vert _{L^{\infty}}+\eta)\left\Vert A\right\Vert _{L^{\infty}}\left((T+T^{2})\kappa T+I_{0}T^{2}\right).\label{Term3}
\end{eqnarray}

\noindent Combining (\ref{GuCu}), (\ref{Term1}), (\ref{Term2})
and (\ref{Term3}), we end up with an estimate of the form

\[
\left|V_{g}\tilde{f}(t,\eta)-V_{g}f(t,\eta)\right|\leq C_{4}(T+T^{2}(\eta+1)),
\]

\noindent for some $C_{4}=C_{4}(A_{k},\phi_{k},g)$ independent of
$T$, $\{a_{n}\}$ or $\eta$. This finally implies the result (\ref{L2E1}).
For the derivative inequality (\ref{L2E2}), we can simply use the
result of (\ref{L2E1}) with $g$ replaced by $g'$, which shows that
for some $C_{5}$,
\begin{eqnarray*}
 &  & \frac{1}{2\pi}\left|\partial_{t}V_{g}\tilde{f}(t,\eta)-\partial_{t}V_{g}f(t,\eta)\right|\\
 & \leq & \eta\sum_{k=1}^{K}\bigg(\left|\int_{-\infty}^{\infty}\partial_{u}(A_{k}(u)g(u-t))e^{2\pi i(\phi_{k}(u)-\eta(u-t))})C(u)du\right|+\\
 &  & \quad\frac{1}{2\pi}\left|\int_{-\infty}^{\infty}\partial_{u}(A_{k}(u)g'(u-t)e^{2\pi i(\phi_{k}(u)-\eta(u-t))})C(u)du\right|\bigg)\\
 & \leq & C_{5}(T\eta+T^{2}\eta^{2}+T^{2}\eta).
\end{eqnarray*}

\end{proof}
\ 
\begin{proof}[Proof of Theorem \ref{ThmMain}]
 For the first part of Theorem \ref{ThmMain}, we suppose that $(t,\eta)\in Z_{k}$
for some $k$, $|V_{g}\tilde{f}(t,\eta)|>E_{1}+E_{2}$ and $0\leq\eta\leq\frac{1}{T}$.
Putting Lemmas \ref{ThmMainP1} and \ref{ThmMainP2} together, we
get $\left|V_{g}\tilde{f}(t,\eta)-Q_{k}(t,\eta)\right|\leq E_{1}+E_{2}$
and $\left|\frac{1}{2\pi i}\partial_{t}V_{g}\tilde{f}(t,\eta)-\phi_{k}'(t)Q_{k}(t,\eta)\right|\leq E_{1}'+E_{2}'$.
We conclude that
\begin{eqnarray}
 &  & |\omega\tilde{f}(t,\eta)-\phi_{k}'(t)|\nonumber \\
 & = & \left|\frac{\frac{1}{2\pi i}\partial_{t}V_{g}\tilde{f}(t,\eta)-\phi_{k}'(t)Q_{k}(t,\eta)}{V_{g}\tilde{f}(t,\eta)}+\frac{\phi_{k}'(t)(Q_{k}(t,\eta)-V_{g}\tilde{f}(t,\eta))}{V_{g}\tilde{f}(t,\eta)}\right|\nonumber \\
 & < & \frac{E_{1}'+E_{2}'+E_{3}(E_{1}+E_{2})}{E_{1}+E_{2}}\nonumber \\
 & \leq & \frac{\alpha}{2}.\label{OmegaBound}
\end{eqnarray}

For the second part of the theorem, let $t$ be fixed and suppose
$\xi\in\alpha\mathbb{N}$, $\xi\notin\bigcup_{1\leq k\leq K}\{\lfloor\frac{\phi_{k}'(t)}{\alpha}\rfloor\alpha,\lceil\frac{\phi'_{k}(t)}{\alpha}\rceil\alpha\}$.
If $(t,\eta)\notin Z_{k}$ for any $k$, then by Lemmas \ref{ThmMainP1}
and \ref{ThmMainP2} we would have $|V_{g}\tilde{f}(t,\eta)|\leq E_{1}+E_{2}<\gamma$.
But since $|V_{g}\tilde{f}(t,\eta)|\geq\gamma$ by assumption, $(t,\eta)\in Z_{k}$
for some $k$. If $\xi>\lceil\frac{\phi_{k}'(t)}{\alpha}\rceil\alpha$,
then the estimate (\ref{OmegaBound}) gives
\[
|\xi-\omega\tilde{f}(t,\eta)|\geq\xi-|\omega\tilde{f}(t,\eta)|\geq\xi-\phi'_{k}(t)-\frac{\alpha}{2}\geq\frac{\alpha}{2},
\]
and similarly, if $\xi<\lfloor\frac{\phi_{k}'(t)}{\alpha}\rfloor\alpha$,
we get $|\xi-\omega\tilde{f}(t,\eta)|\geq|\omega\tilde{f}(t,\eta)|-\xi\geq\frac{\alpha}{2}$.
This means that $\{\eta:~|\xi-\omega\tilde{f}(t,\eta)|<\frac{\alpha}{2},~|V_{g}\tilde{f}(t,\eta)|\geq\gamma,~0\leq\eta\leq\frac{1}{T}\}=\emptyset$
and so $S^{\alpha,\gamma}\tilde{f}(t,\xi)=0$.\end{proof}
\begin{rem*}
We have studied functions $f\in\mathcal{B}_{\epsilon,d}$ on the entire
real line here, but we can instead consider such functions supported
only on a finite interval, and the same results will hold with largely
only notational changes. We can also consider real-valued functions
$f$ that have the exponentials $e^{2\pi i\phi_{k}(t)}$ in (\ref{IMF})
replaced by $\cos(2\pi\phi_{k}(t))$, and similar results will hold.
The assumption that $T\leq1$ in Theorem \ref{ThmMain} was made for
convenience in the proof, but it poses no loss of generality since
lower sampling rates are equivalent to simply rescaling $f$ (and
thus also its IIF).
\end{rem*}
\ 
\begin{rem*}
In Theorem \ref{ThmMain}, we required the window $g$ to be bandlimited
to $[-\frac{d}{2},\frac{d}{2}]$. This assumption is not strictly
necessary and was made here to simplify the presentation. In general,
it is enough for $g\in\mathcal{S}$ to be such that $|\hat{g}|$ is
small outside $[-\frac{d}{2},\frac{d}{2}]$. The Gaussian window $g(t)=e^{-\pi t^{2}d^{2}/4}$
works well in practice. For such a window, there will be an extra
term in the error estimates in Lemma \ref{ThmMainP1}, but the results
still remain essentially the same. (e.g. if we assume $\left\Vert \hat{g}\right\Vert _{L^{\infty}(\mathbb{R}\backslash[-\frac{d}{2},\frac{d}{2}])}\leq\epsilon$,
then instead of having $Q_{l}(t,\eta)=0$ for $(t,\eta)\not\in Z_{l}$,
we would get $|Q_{l}(t,\eta)|\leq C\epsilon$ for a constant $C$.)
\end{rem*}
We finally show that if $\alpha$ and $\gamma$ are chosen properly,
the calculation of $\mathrm{IF}_{S}$ is robust to sample noise. The
argument is almost the same as in the proof of Theorem \ref{ThmMain}.

\begin{theorem}\label{ThmRobust}Let $0\leq T\leq1$. Suppose $\tilde{f}\in\mathcal{D}_{\epsilon,d}^{T,\{a_{n}\}}$
is as in (\ref{ftilde}) and let $g$ be a window with the same conditions
as in Theorem \ref{ThmMain}. Assume that the samples $\{f(t_{n})\}=\{f(Tn+a_{n})\}$
are contaminated by noise $\{N_{n}\}$ with $\left\Vert \{N_{n}\}\right\Vert _{l^{\infty}}\leq T$,
i.e. we are given $\bar{f}=\tilde{f}+\sum_{n=-\infty}^{\infty}(t_{n+1}-t_{n})\delta(\cdot-t_{n})N_{n}.$
Suppose we have a resolution $\alpha$ satisfying
\[
\alpha\geq\frac{2(E_{1}'+E_{2}'+I_{0}+2I_{0}'+\frac{1}{2\pi}(I_{0}'+2I_{0}''))}{E_{1}+E_{2}+I_{0}+2I_{0}'}+2E_{3},
\]
where $E_{2}$ and $E_{2}'$ are as in Lemma \ref{ThmMainP2}. Then
the following statements hold.
\begin{enumerate}
\item Let $0\leq\eta\leq\frac{1}{T}$ and fix $k$, $1\leq k\leq K$. For
each pair $(t,\eta)\in Z_{k}$ with $|V_{g}\bar{f}(t,\eta)|>E_{1}+E_{2}+I_{0}+2I_{0}'$,
we have $|\omega\bar{f}(t,\eta)-\phi_{k}'(t)|<\frac{\alpha}{2}.$
If $(t,\eta)\not\in Z_{k}$ for any $k$, then $|V_{g}\bar{f}(t,\eta)|\leq E_{1}+E_{2}+I_{0}+2I_{0}'$.
\item Suppose we have a threshold $\gamma$ such that $E_{1}+E_{2}+I_{0}+2I_{0}'<\gamma\leq|V_{g}\bar{f}(t,\eta)|$
for all $(t,\eta)\in Z_{k}$. Then for all $t$, $S^{\alpha,\gamma}\bar{f}(t,\xi)$
is supported in the $2K$-point set $\bigcup_{1\leq k\leq K}\{\lfloor\frac{\phi_{k}'(t)}{\alpha}\rfloor\alpha,\lceil\frac{\phi'_{k}(t)}{\alpha}\rceil\alpha\}$.
\end{enumerate}
\end{theorem}

\noindent The only part of the proof different from Theorem \ref{ThmMain}
is the estimate in Lemma \ref{ThmMainP2}, which we replace by the
following inequalities.

\begin{lemma}Let $0\leq\eta\leq\frac{1}{T}$. Then 
\[
\left|V_{g}\bar{f}(t,\eta)-V_{g}f(t,\eta)\right|\leq E_{2}+I_{0}+2I_{0}'
\]
and
\[
\frac{1}{2\pi}\left|\partial_{t}V_{g}\bar{f}(t,\eta)-\partial_{t}V_{g}f(t,\eta)\right|\leq E_{2}'+I_{0}+2I_{0}'+\frac{1}{2\pi}(I_{0}'+2I_{0}'').
\]

\end{lemma}
\begin{proof}
Lemma \ref{ThmMainP2} and the inequality (\ref{hBound}) imply that
\begin{eqnarray*}
 &  & \left|V_{g}\bar{f}(t,\eta)-V_{g}f(t,\eta)\right|\\
 & \leq & \left|V_{g}\tilde{f}(t,\eta)-V_{g}f(t,\eta)\right|+\left\Vert N_{n}\right\Vert _{l^{\infty}}\sum_{n=-\infty}^{\infty}\left|(t_{n+1}-t_{n})g(t_{n}-t)\right|\\
 & \leq & E_{2}+T(I_{0}+2TI_{0}'),
\end{eqnarray*}
and similarly,
\begin{eqnarray*}
\frac{1}{2\pi}\left|\partial_{t}V_{g}\bar{f}(t,\eta)-\partial_{t}V_{g}f(t,\eta)\right| & \leq & E_{2}'+\eta T(I_{0}+2TI_{0}')+\frac{1}{2\pi}T(I_{0}'+2TI_{0}'').
\end{eqnarray*}

\end{proof}
\noindent Combining these bounds with the other estimates in the proof
of Theorem \ref{ThmMain} gives Theorem \ref{ThmRobust}.

\section{\label{SecBL}A Bandlimited Reconstruction Approach}

We now consider a somewhat different formulation and approach towards
our problem, based on the traditional $\mathrm{IF}_{H}$ concept.
Suppose we have the samples $\{f(t_{k})\}$ of a function $f\in L^{2}$
with $\mbox{supp(}\hat{f})\subset[-b,b]$. We want to determine $f$
and use it to find $\mathrm{IF}_{H}f$. In the case of uniformly spaced
samples, i.e. $t_{k}=Tk$ for a constant $T>0$, it is well known
that $f$ can be recovered if the sampling rate $\frac{1}{T}$ exceeds
$2b$, the Nyquist frequency of $f$. There are analogous results
for very general classes of nonuniform sampling points $\{t_{k}\}$
\cite{AG01,Ma01}. As in Section \ref{SecMain}, we will consider
sampling points $\{t_{k}\}=\{Tk+a_{k}\}$ that are small perturbations
of uniformly spaced points. Suppose the sequence $\{t_{k}\}$ is indexed
so that $t_{k}<t_{k+1}$ and that for some $T>0$,
\begin{equation}
\sup_{k}\left|t_{k}-Tk\right|<\frac{T}{2}.\label{supXk}
\end{equation}
If $\frac{1}{T}>2b$, it can be shown that $C_{1}\left\Vert f\right\Vert _{L^{2}}\leq\left\Vert f(t_{k})\right\Vert _{l^{2}}\leq C_{2}\left\Vert f\right\Vert _{L^{2}}$,
where the constants $C_{1}$ and $C_{2}$ depend only on $\{t_{k}\}$
and $b$ \cite{DS52,Gr99}. When the $t_{k}$ are uniformly spaced,
i.e. $t_{k}=Tk$, the Plancherel formula shows that $C_{1}=C_{2}=T^{-1/2}$.
For the rest of this section, we assume that $\{t_{k}\}$ satisfies
(\ref{supXk}) with $\frac{1}{T}>2b$, which intuitively means that
the sampling points $\{t_{k}\}$ are evenly spread out over $\mathbb{R}$
and have an {}``average'' sampling rate above the Nyquist rate.\\

One of the standard approaches for recovering $f$ from the samples
$f(t_{k})$ involves solving the linear least-squares problem
\begin{equation}
\{\tilde{c}_{k,N}\}=\underset{\{c_{k}\}\in\mathbb{C}^{2N+1}}{\operatorname{argmin}}\left\Vert w_{k}\left(f(t_{k})-\sum_{n=-N}^{N}c_{n}h\left(n,t_{k}/T\right)\right)\right\Vert _{l^{2}}^{2},\label{ckN}
\end{equation}
where $\{h(n,t)\}_{n\in[-N,N]}$ is a given set of basis functions
and the $w_{k}$ are some weights such that $|w_{k}|\eqsim1$ for
all $k$ \cite{Gr99,Ma01,HGS95}. If the basis and weights are chosen
appropriately, the function
\begin{equation}
f_{N}(t)=\sum_{n=-N}^{N}\tilde{c}_{n,N}h(n,t)\label{fN}
\end{equation}
is a good approximation to $f_{T}(t):=f(Tt)$. There have been a couple
of efficient algorithms developed around this idea, usually using
the DFT basis $h(n,t)=(2N+1)^{-1/2}e^{-2\pi int/(2N+1)}$ and the
weights $w_{k}=\frac{t_{k+1}-t_{k-1}}{2}$. For these choices, as
$N\to\infty$, $f_{N}$ converges uniformly to $f_{T}$ on compact
subsets of $\mathbb{R}$, and the matrix computations used in solving
the problem (\ref{ckN}) are well-conditioned and can be accelerated
using FFTs \cite{Gr99}. Once we have a expansion of the form (\ref{fN}),
it follows by linearity that $P^{+}f_{N}(t)=\sum_{n=-N}^{N}\tilde{c}_{n,N}P^{+}h(n,t)$,
and we can use this to find $\mathrm{IF}_{H}f$.\\

We will prefer to use the basis $h(n,t)=\mathrm{sinc}(t-n-M)$ in
this paper, where $M$ is an integer constant, along with the same
weights $w_{k}$ as above. In this case, $f_{N}$ converges to $f_{T}$
uniformly on the entire real line, and we have found that in practice,
this results in a better accuracy with the computation of $P^{+}f_{N}$.
This choice of $h(n,t)$ can be justified by the following argument.
For any vector $\{c_{k}\}\in l^{2}$ with $\mathrm{supp}(c_{k})\subset[M-N,M+N]$,
we have
\begin{align*}
\left\Vert f_{T}(k)-c_{k}\right\Vert _{l^{2}} & =\left\Vert f_{T}(t)-\sum_{n=-N}^{N}c_{n+M}\mathrm{sinc}(t-n-M)\right\Vert _{L^{2}}\\
 & \eqsim\left\Vert f(t_{k})-\sum_{n=-N}^{N}c_{n+M}\mathrm{sinc}(t_{k}/T-n-M)\right\Vert _{l^{2}}\\
 & \eqsim\left\Vert w_{k}\left(f(t_{k})-\sum_{n=-N}^{N}c_{n+M}\mathrm{sinc}(t_{k}/T-n-M)\right)\right\Vert _{l^{2}}.
\end{align*}
By taking the minimum over $\{c_{k}\}$, it follows that for $\{\tilde{c}_{k,N}\}$
given by (\ref{ckN}),
\[
\left\Vert f_{T}(k)-\tilde{c}_{k,N}\right\Vert _{l^{2}}\eqsim\left\Vert f_{T}(k)-f_{T}(k)\chi_{[M-N,M+N]}(k)\right\Vert _{l^{2}}\to0
\]
as $N\to\infty$. We can then conclude that
\[
\left\Vert f_{T}-f_{N}\right\Vert _{L^{\infty}}\leq\left\Vert \hat{f_{T}}-\hat{f_{N}}\right\Vert _{L^{1}}\leq\left\Vert f_{T}-f_{N}\right\Vert _{L^{2}}=\left\Vert f_{T}(k)-\tilde{c}_{k,N}\right\Vert _{l^{2}}\to0.
\]
So we essentially determine uniform samples of $f$ from the nonuniform
samples $\{f(t_{k})\}$ by solving the problem (\ref{ckN}), and use
those in a classical sampling series to compute $f$. In practice,
we consider a finite number of samples spread over some interval $[J_{1},J_{2}]$,
and we set $M=\lfloor\frac{1}{2T}(J_{1}+J_{2})\rfloor$ to center
the basis $h(n,t)$ appropriately.\\

\noindent Once we have recovered the function $f$, we perform some
elementary calculations with Fourier transforms to find that
\[
P^{+}f_{N}(t)=\sum_{n=-N}^{N}\tilde{c}_{n,N}\frac{1}{2}\mathrm{sinc}\left(\frac{t-n-M}{2}\right)e^{\frac{\pi i}{2}(t-n+M)}
\]
and
\[
\frac{d}{dt}P^{+}f_{N}(t)=\sum_{n=-N}^{N}\tilde{c}_{n,N}\frac{(i+\pi(t-n-M))e^{\pi i(t-n+M)}-i}{2\pi(t-n-M)^{2}},
\]
from which we can approximate $\mathrm{IF}_{H}f$. Note that the Synchrosqueezing-based
method discussed in Section \ref{SecMain} determines the IF components
(as defined by $\mathrm{IF}_{S}$) of the function $f$ directly,
while the approach considered in this section treats the function
$f$ as a whole and determines its IF (given by $\mathrm{IF}_{H}$)
after the signal itself has been recovered.\\

We finally make a few comments on the applicability of the framework
discussed in this section. The assumption that $f\in L^{2}$ can be
weakened to $f\in L^{\infty}$ in practice. Suppose $B$ is a function
such that $\hat{B}$ is smooth, $\hat{B}(0)=1$ and $\mathrm{supp}(\hat{B})\subset[-1,1]$.
For $\frac{1}{T}>2b$, any $f\in L^{\infty}$ with $\mathrm{supp}(\hat{f})\subset[-b,b]$
can be expressed as the sampling series
\[
f(t)=\sum_{k=-\infty}^{\infty}f\left(Tk\right)\mathrm{sinc}\left(\frac{t}{T}-k\right)B\left(\left(\frac{1}{T}-2b\right)(t-k)\right),
\]
which converges uniformly on compact sets \cite{Ca68}. If we restrict
our attention to a fixed finite interval $I$, then $f$ can be well
approximated on $I$ by taking a finite part of this series $f_{p}\in L^{2}$,
which has $\mathrm{supp}(\hat{f_{p}})\subset[-(\frac{1}{T}-b),(\frac{1}{T}-b)]$.
As long as $f$ is oversampled, our reconstruction method will approximately
recover $f_{p}$. In the next section, we will apply the least-squares
method to AM-FM functions $f$ of the type in Definition \ref{DefBClass},
which are generally not strictly bandlimited, but by considering Fourier
series on a given finite interval $I$, such functions can be closely
approximated on $I$ by bandlimited $L^{\infty}$ functions that have
a sufficiently high bandwidth (see \cite{Sc10} for details). In practice,
if our sampling interval $T$ is greater than $\sup_{k,t\in I}\mathrm{IIF}(f,A_{k},\phi_{k})$,
we can recover $f$ effectively.

\section{\label{SecNumerical}Numerical Experiments and Applications}

We use the algorithms discussed in Sections \ref{SecMain} and \ref{SecBL}
on several test cases. We consider an AM-FM signal, a chirp signal,
a bandlimited Bessel function, a Fourier harmonic contaminated by
noise, an undersampled harmonic, a multi-component signal and finally,
a signal with interlacing IIF elements. These are respectively shown
in Figures \ref{Fig1}-\ref{Fig7} below. We compute the $\mathrm{IF}_{S}$
of these signals using STFT Synchrosqueezing and the $\mathrm{IF}_{H}$
using the bandlimited reconstruction method. For the $\mathrm{IF}_{S}$
computation, the Gaussian window function $g(t)=e^{-0.1\pi t^{2}}$
and the parameters $\alpha=0.1$ and $\gamma=10^{-8}$ work well in
practice, and we use these values unless stated otherwise.\\

In all of these examples, we use sampling times that are random perturbations
of uniformly spaced times and have the form $t_{n}=Tn+T'a_{n}$, where
$T'<T$ and $\{a_{n}\}$ is a fixed realization of a white noise process
uniformly distributed in $[0,1]$. For purposes of comparison, we
also test most of these examples with uniform samples as well, i.e.
$T'=0$. In the figures below, the first two images are respectively
$\mathrm{IF}_{S}f$ and $\mathrm{IF}_{H}f$ determined from uniform
samples, while the two latter images are $\mathrm{IF}_{S}f$ and $\mathrm{IF}_{H}f$
for the general case of nonuniform samples.\\

\begin{figure}[H]
\noindent \centering{}\subfloat{\includegraphics[width=0.22\textwidth]{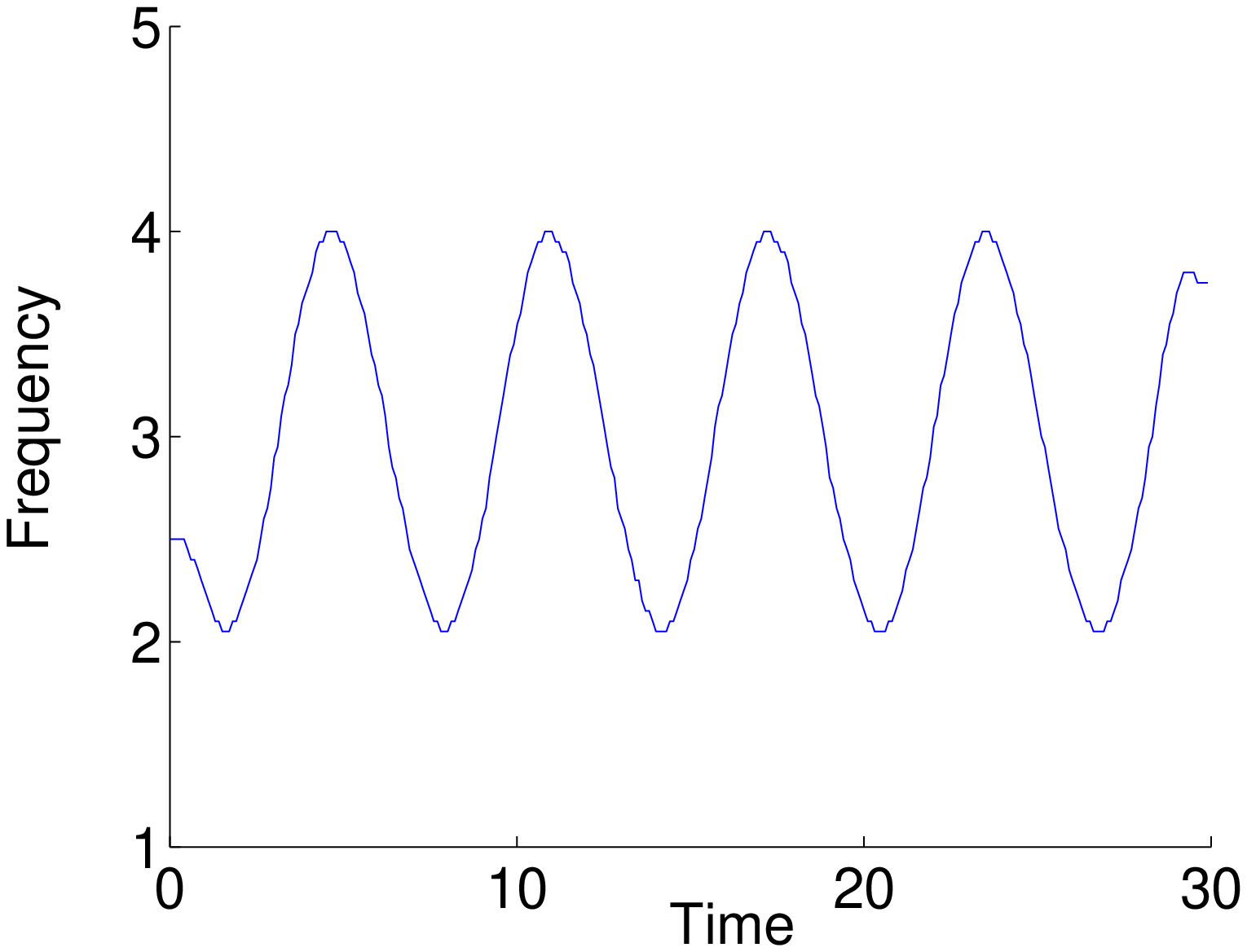}}\subfloat{\includegraphics[width=0.22\textwidth]{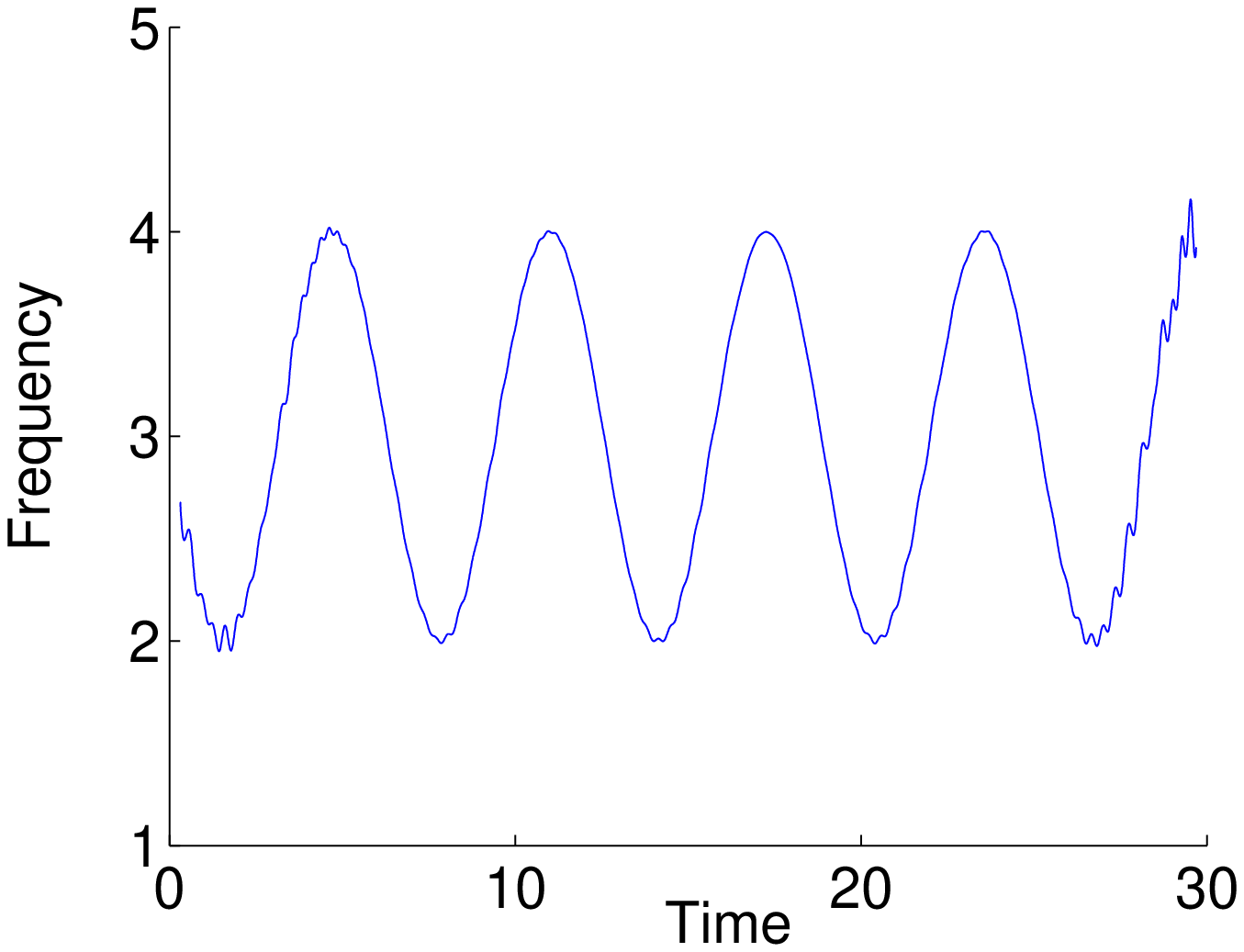}}~~~~~~\subfloat{\includegraphics[width=0.22\textwidth]{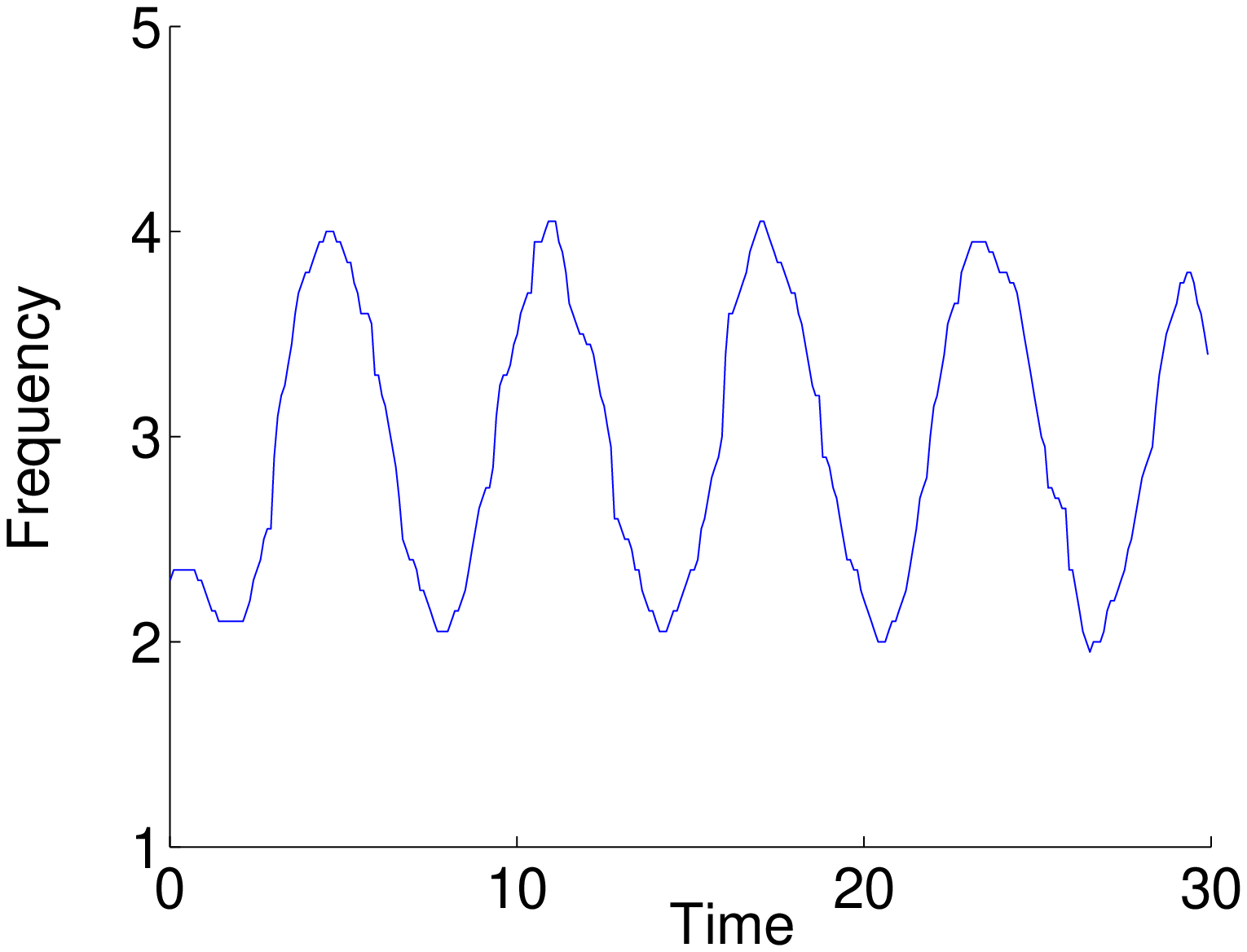}}\subfloat{\includegraphics[width=0.22\textwidth]{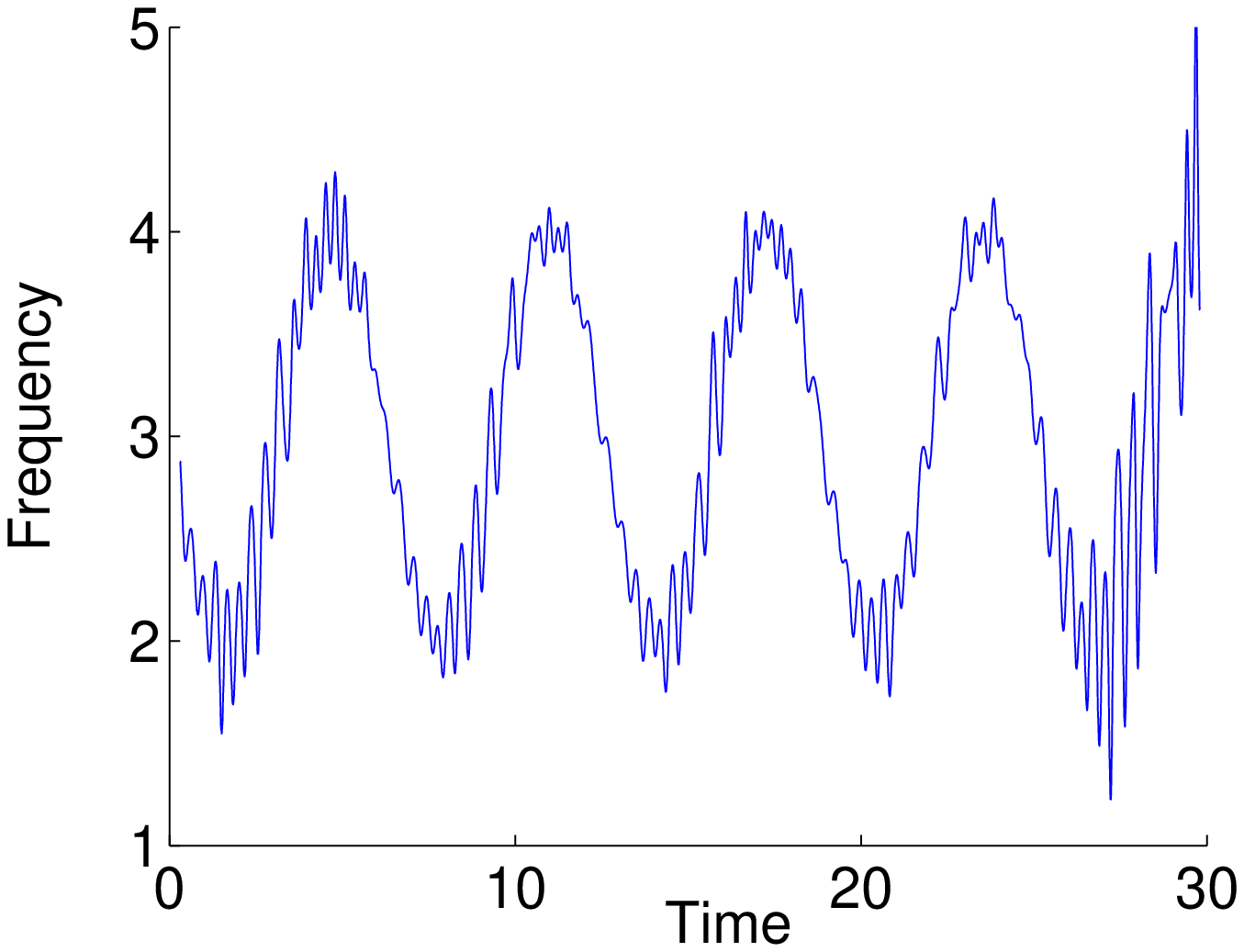}}
\caption{\label{Fig1} The AM-FM signal $f(t)=(2+\cos t)\cos(2\pi(3t+\cos t))$,
$t\in[0,30]$. The samples are taken at $t_{n}=0.1n+T'a_{n}$, $T'=0$
(left images) and $T'=0.08$ (right images). The IIF is $3-\sin t$,
and both methods produce results that are reasonably close to this.}
\end{figure}

\begin{figure}[H]
\centering{}\subfloat{\includegraphics[width=0.22\textwidth]{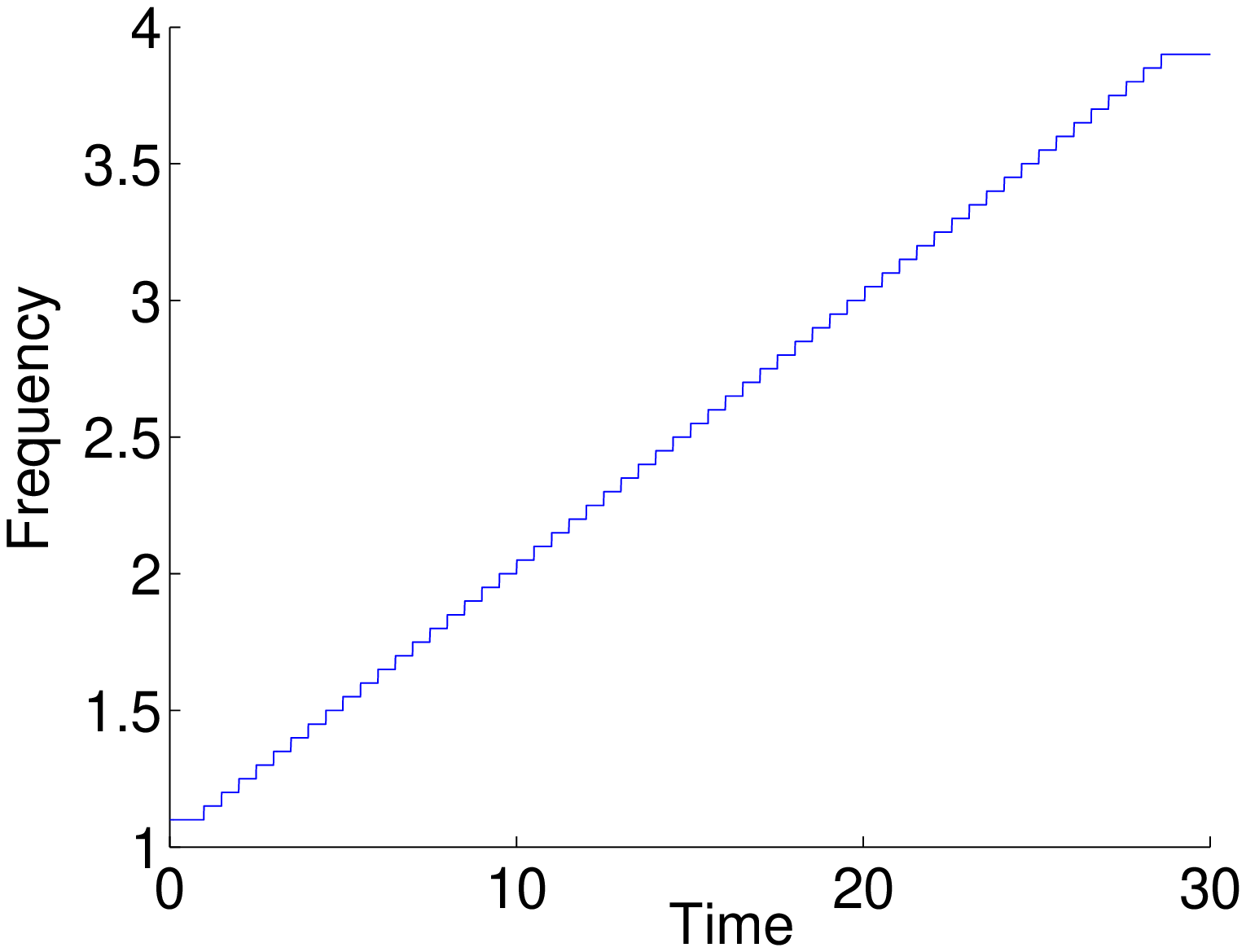}}\subfloat{\includegraphics[width=0.22\textwidth]{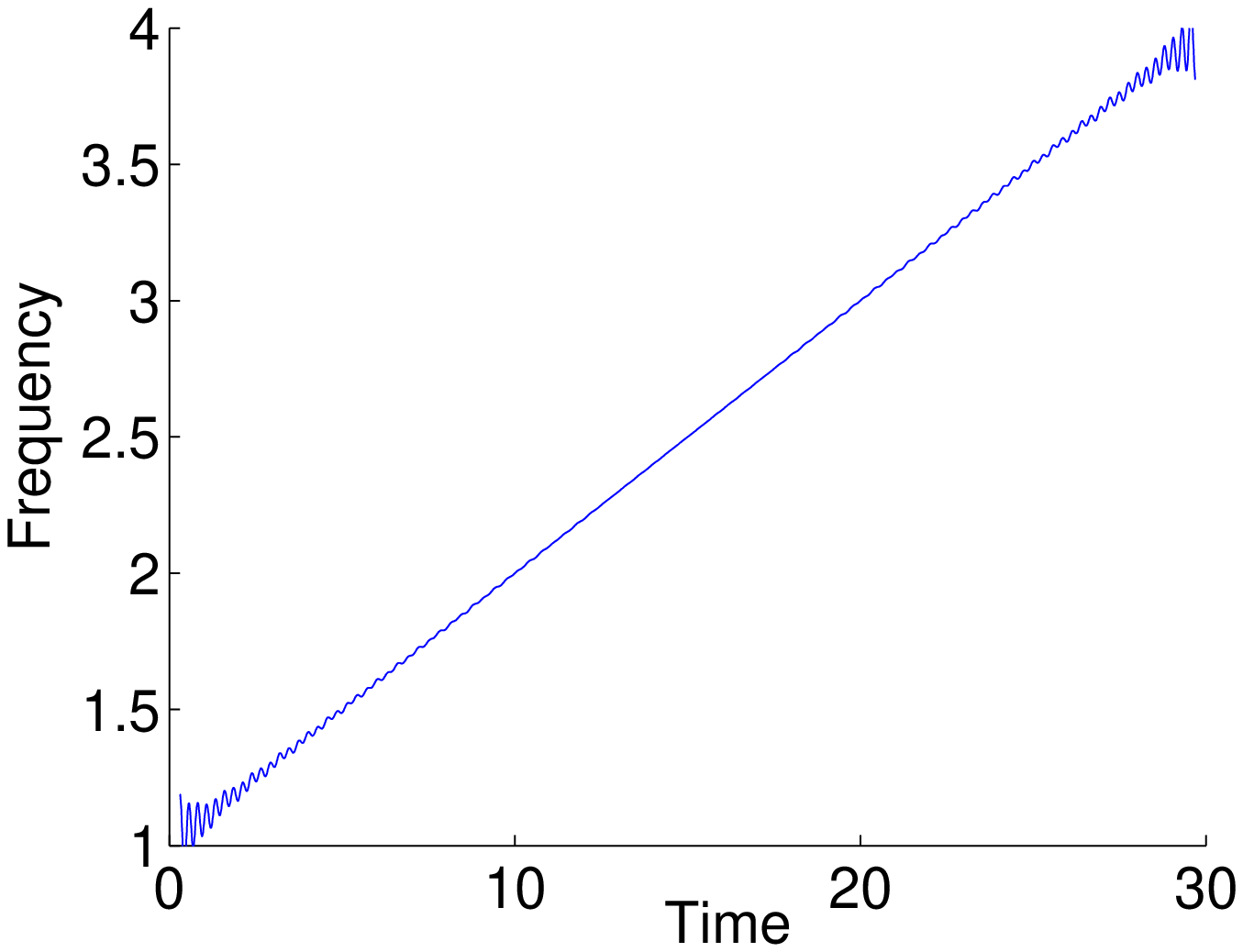}}~~~~~~\subfloat{\includegraphics[width=0.22\textwidth]{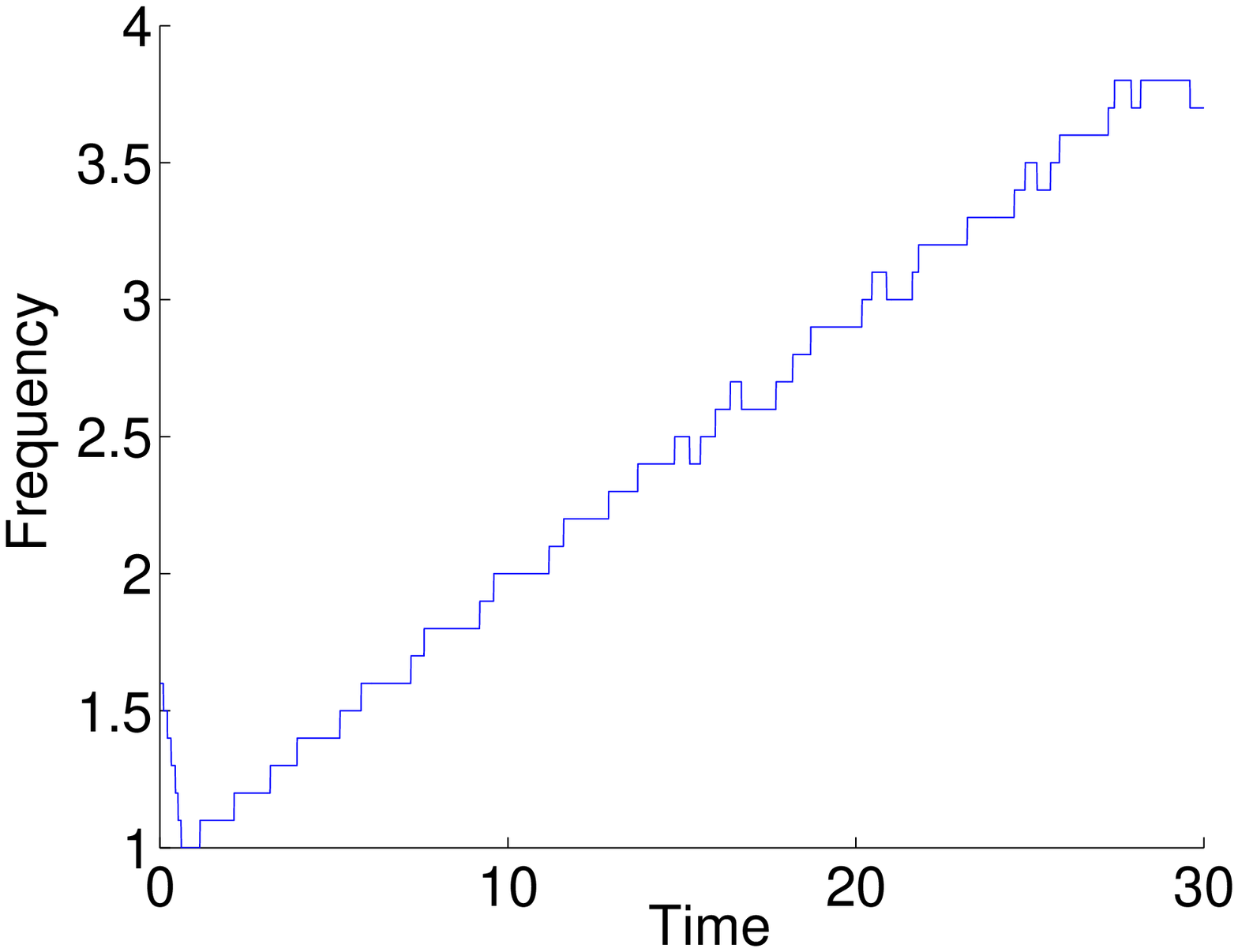}}\subfloat{\includegraphics[width=0.22\textwidth]{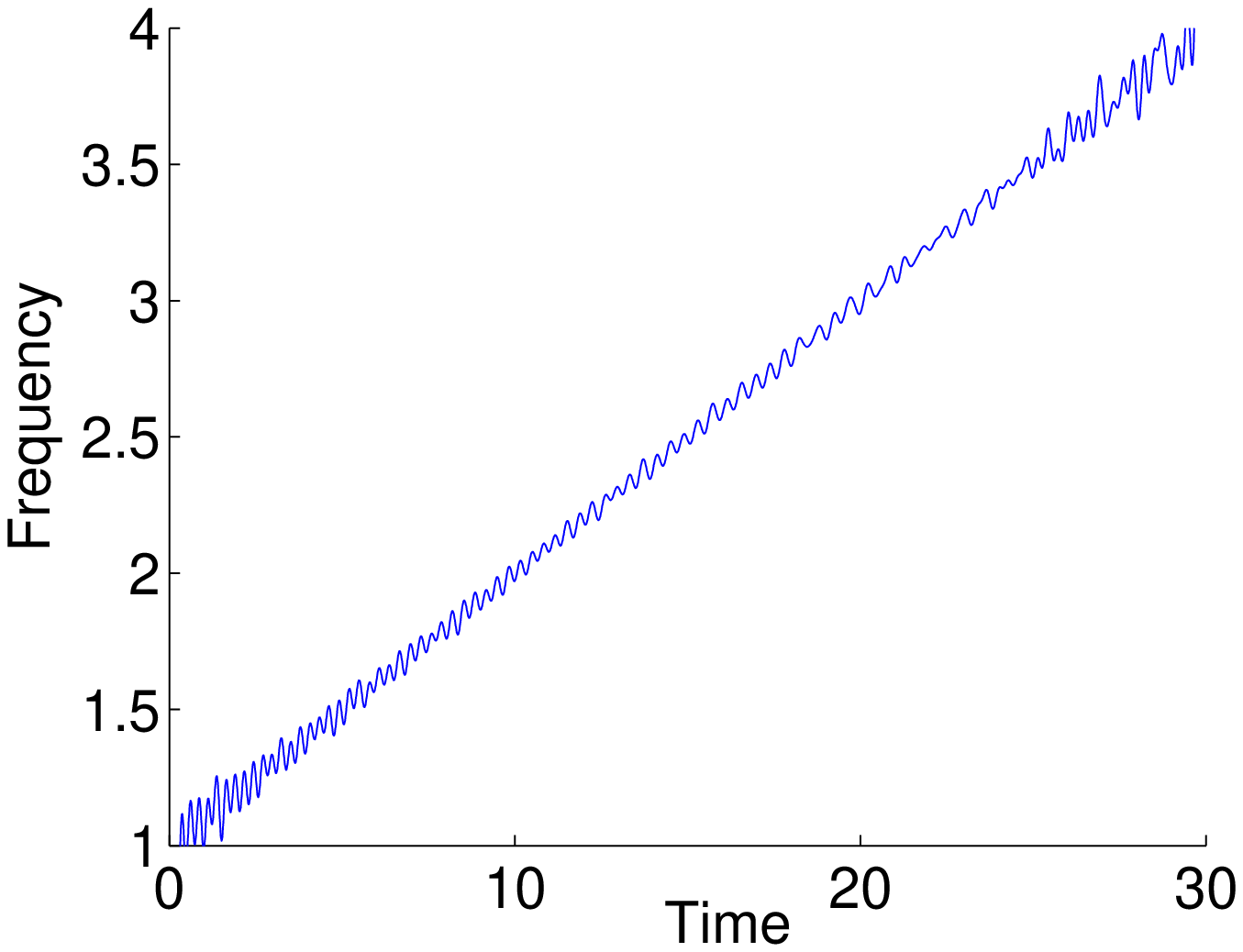}}
\caption{\label{Fig2} The chirp signal $f(t)=\cos(2\pi(t+0.05t^{2}))$, $t\in[0,30]$,
with samples taken at $t_{n}=0.1n+T'a_{n}$, $T'=0$ (left images)
and $T'=0.08$ (right images). The IIF in this case is $1+0.1t$,
and both methods produce results that are very close to this.}
\end{figure}

\begin{figure}[H]
\begin{centering}
\subfloat{\includegraphics[width=0.22\textwidth]{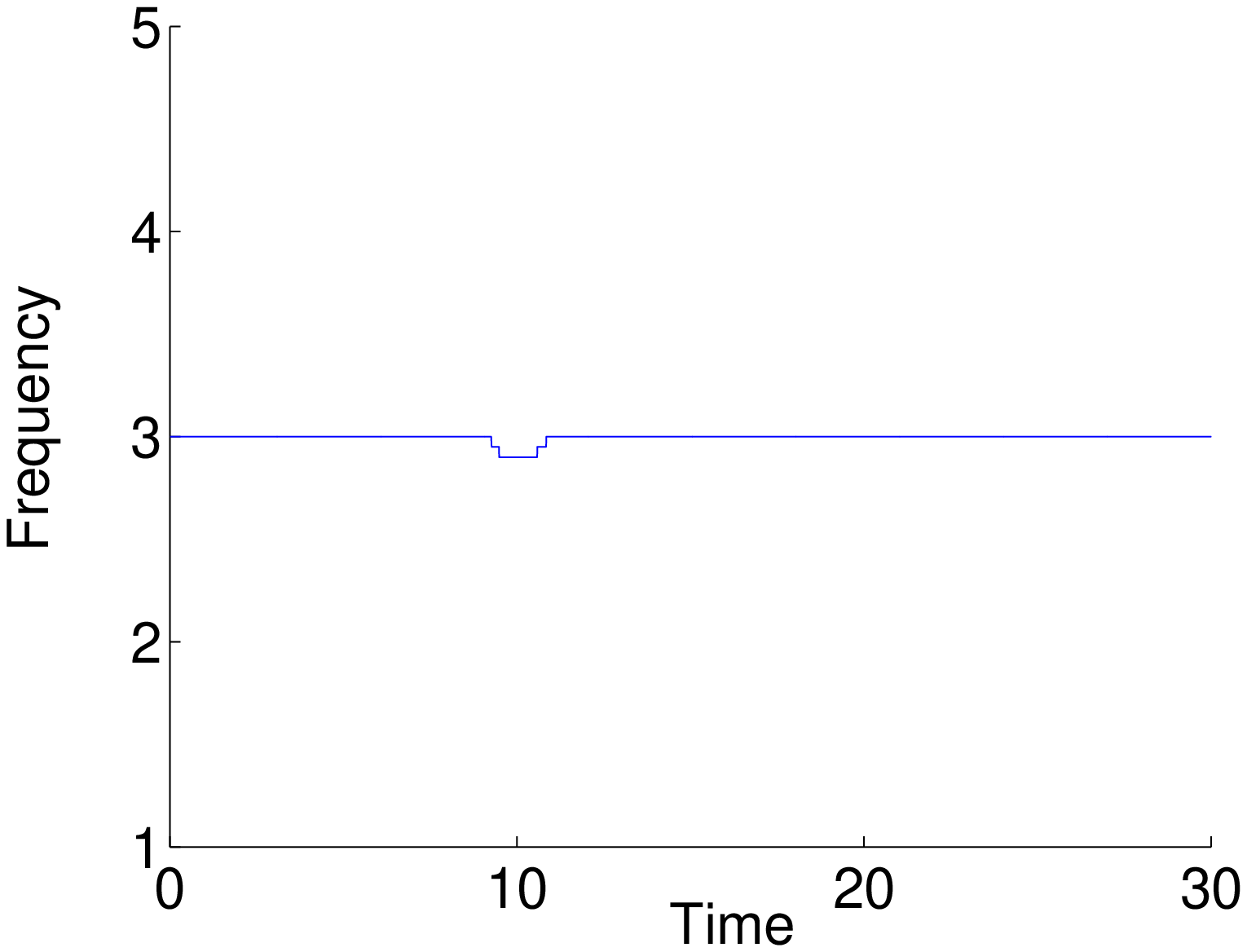}}\subfloat{\includegraphics[width=0.22\textwidth]{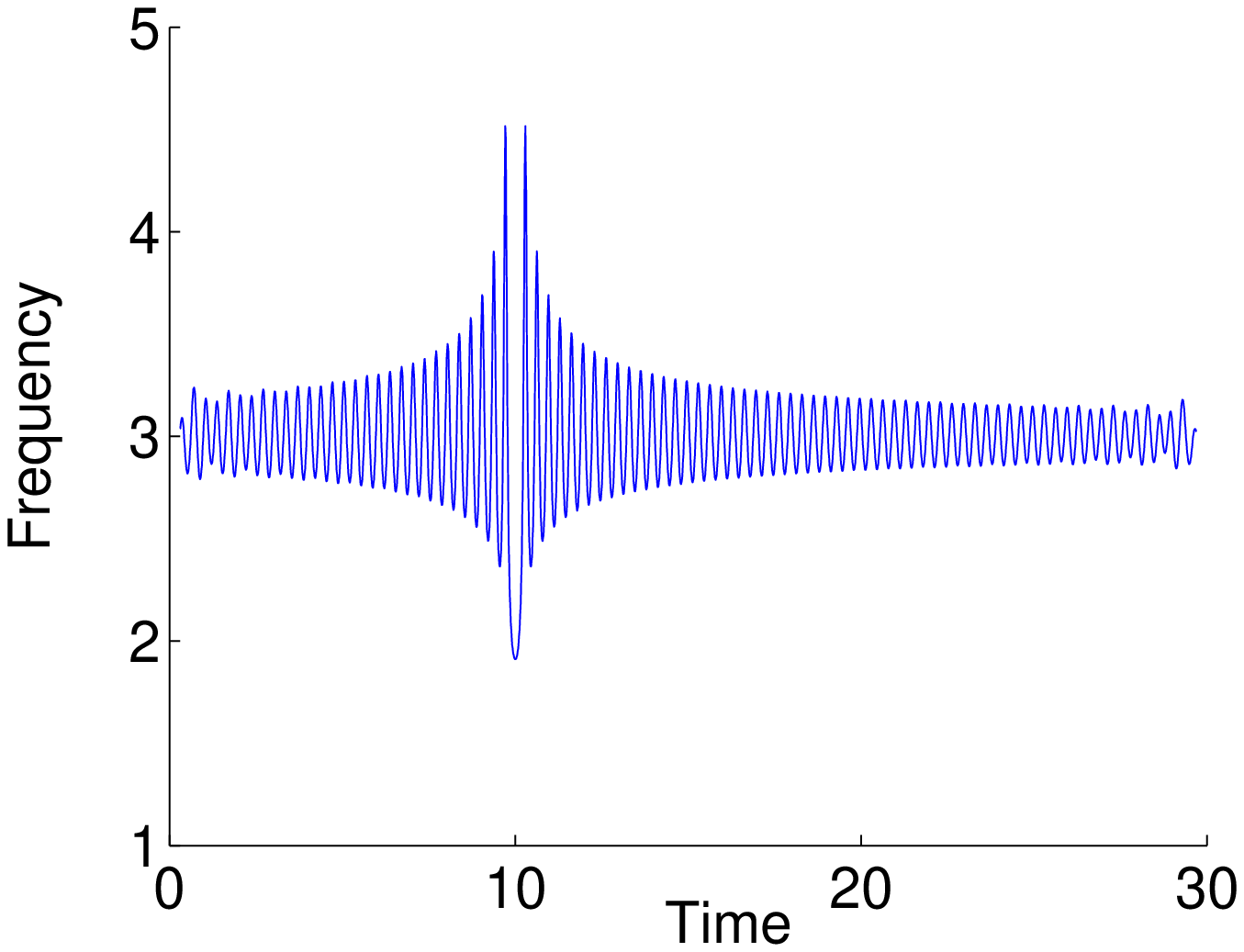}}~~~~~~\subfloat{\includegraphics[width=0.22\textwidth]{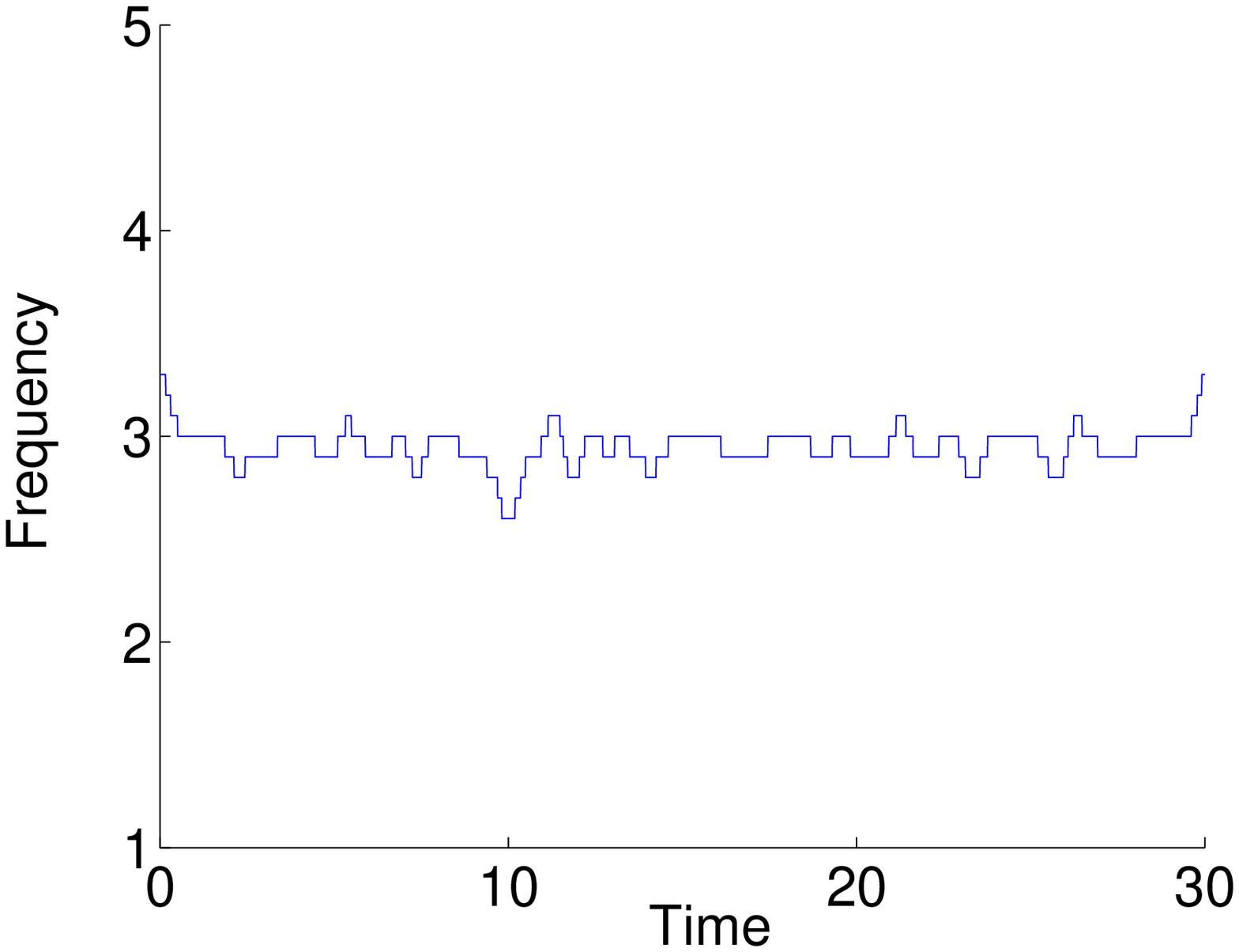}}\subfloat{\includegraphics[width=0.22\textwidth]{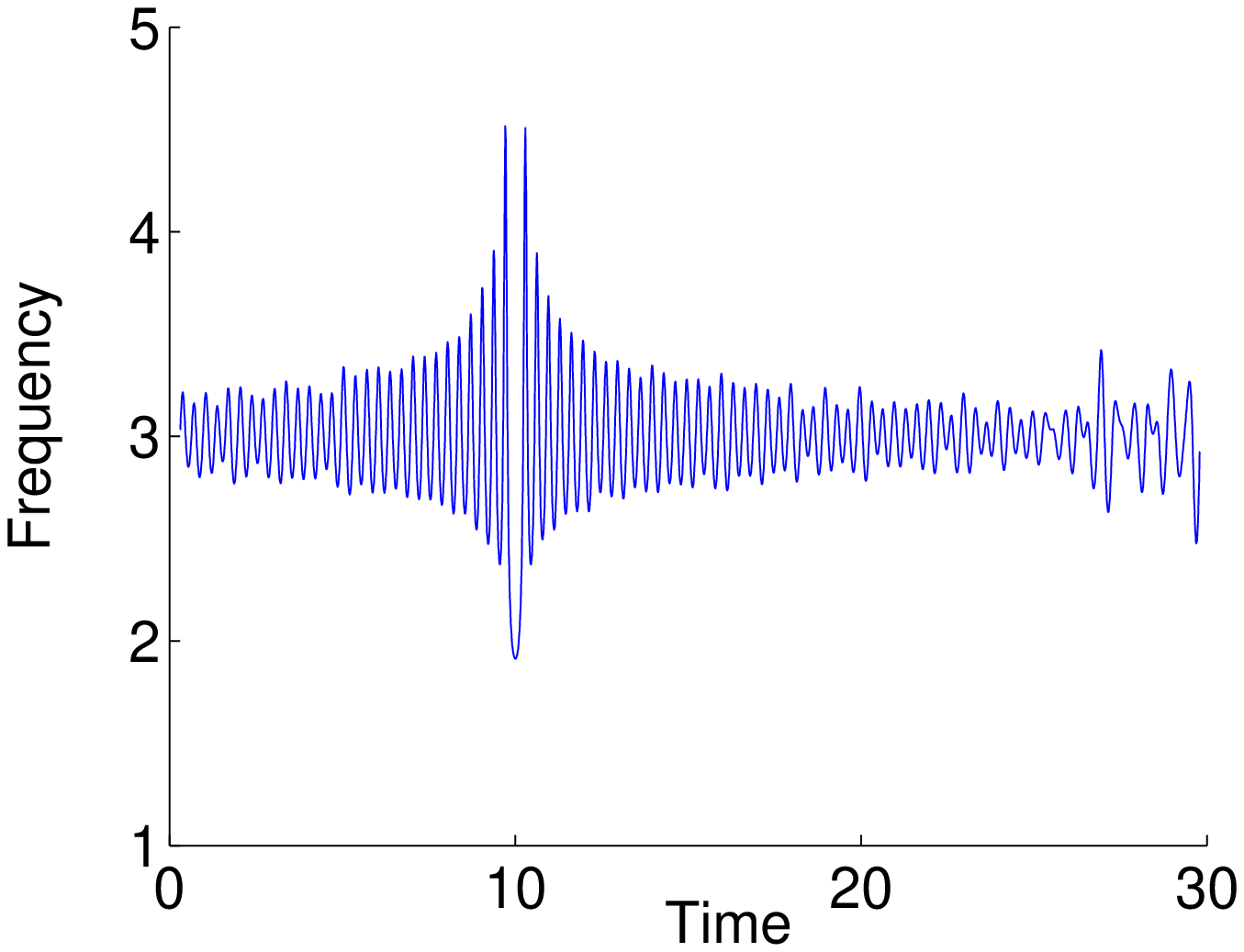}} 
\par\end{centering}

\begin{centering}
\subfloat{\includegraphics[scale=0.5]{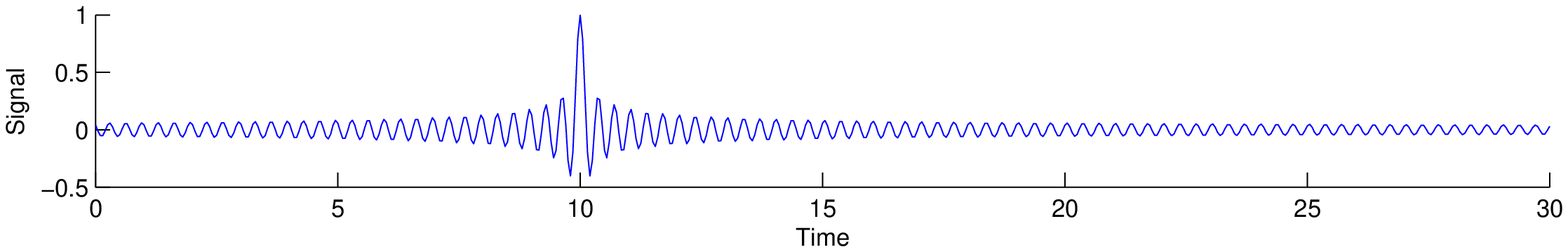}} 
\par\end{centering}

\caption{\label{Fig3} The bandlimited signal $f(t)=J_{0}(6\pi(t-10))$, $t\in[0,30]$,
where $J_{0}$ is the Bessel function of order $0$ (see \cite{WW27}).
The samples are taken at $t_{n}=0.1n+T'a_{n}$, $T'=0$ (left images)
and $T'=0.08$ (right images). There is no way to determine the IIF
in this case, but it can be shown that $J_{0}(0)=1$ and $J_{0}(t)\approx\sqrt{\frac{2}{\pi|t|}}\cos(|t|-\frac{\pi}{4})$
when $|t|$ is large, so we would expect the IF to be roughly the
constant $3$. The computed $\mathrm{IF}_{S}f$ agrees with our intuition
here. On the other hand, $\mathrm{IF}_{H}f$ is highly oscillating,
particularly around $t=10$, which indicates that it is unable to
clearly distinguish between the amplitude and frequency factors of
$f$. We also show the graph of $f$ itself at the bottom for clarity.}
\end{figure}

\begin{figure}[H]
\centering{}\subfloat{\includegraphics[width=0.22\textwidth]{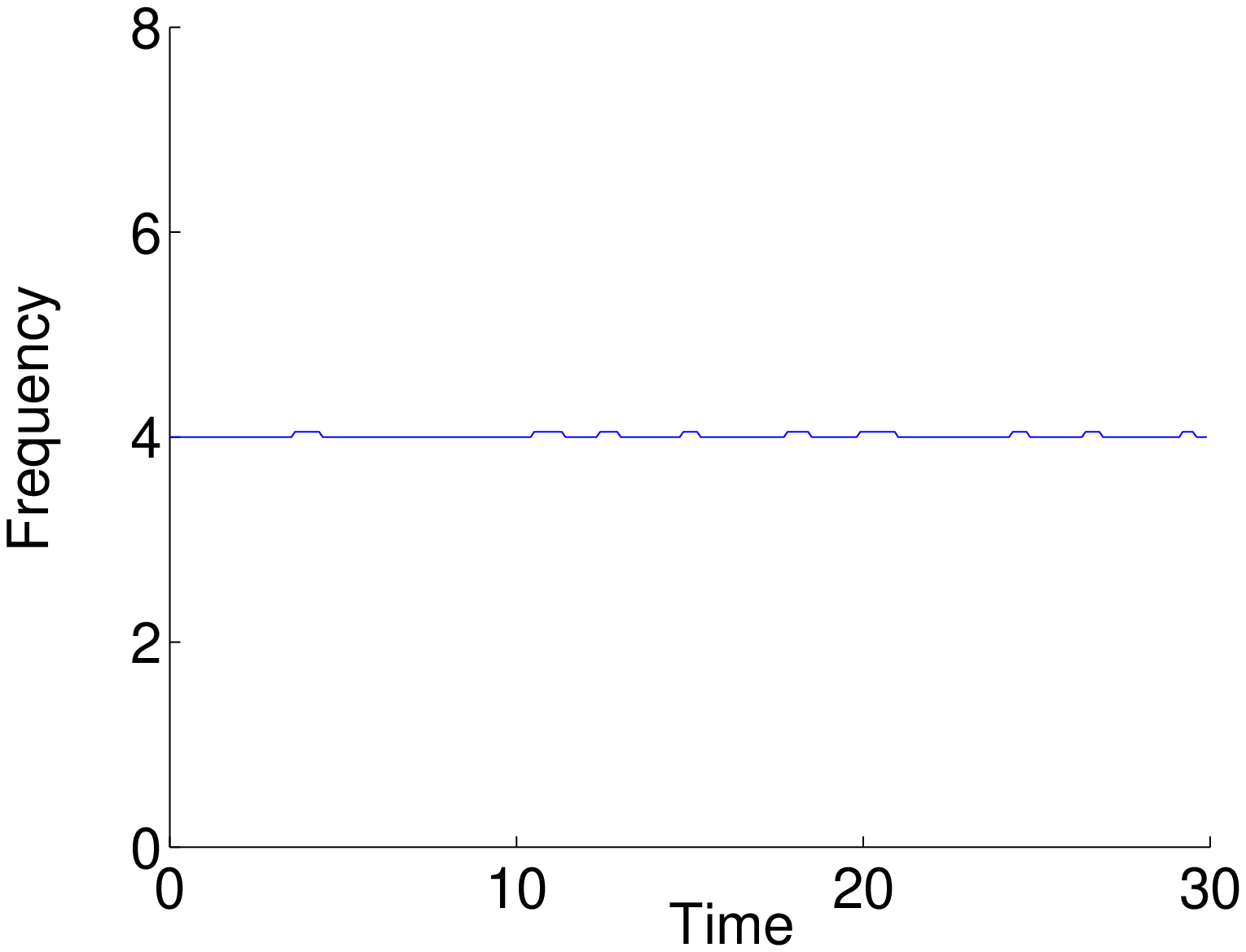}}\subfloat{\includegraphics[width=0.22\textwidth]{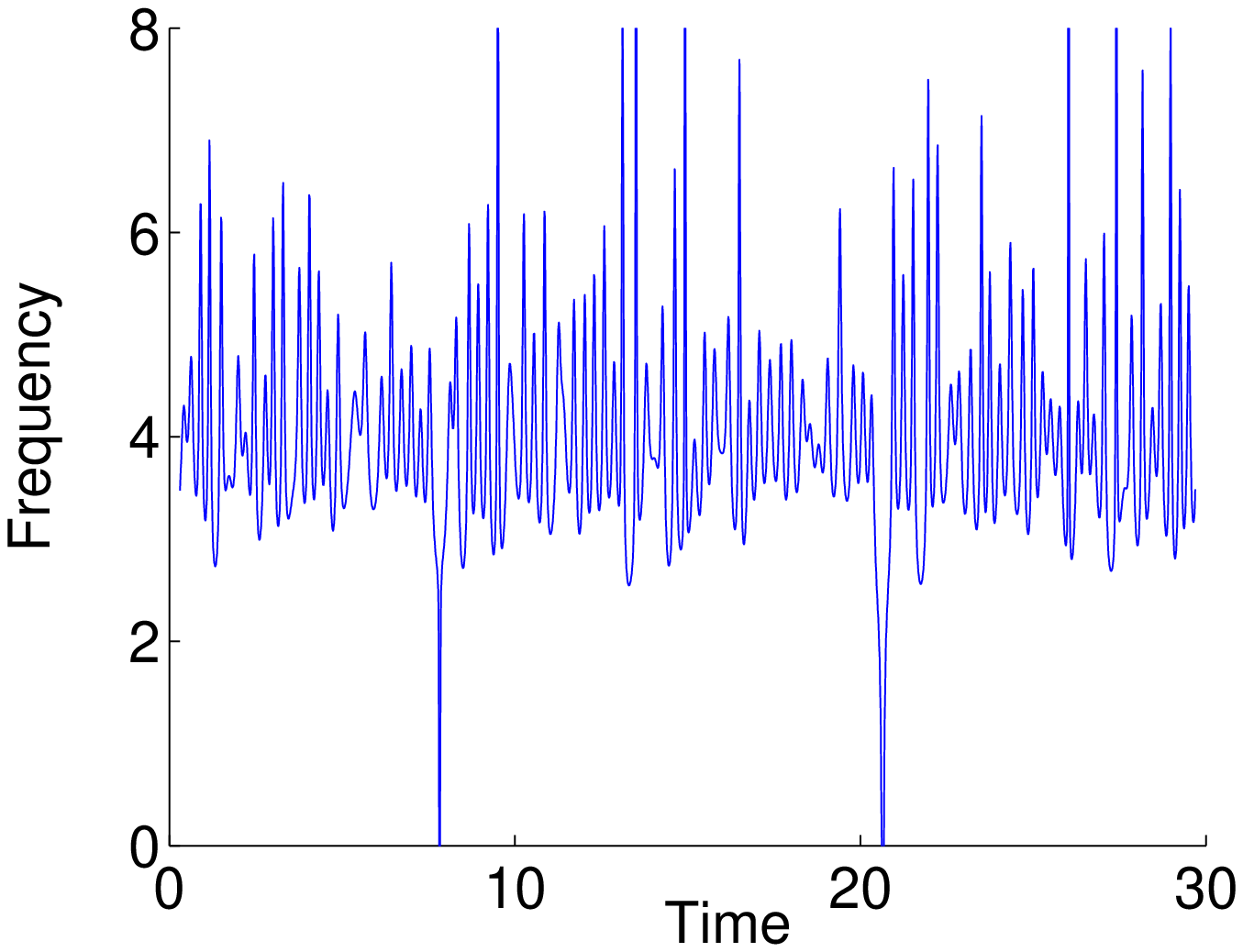}}~~~~~~\subfloat{\includegraphics[width=0.22\textwidth]{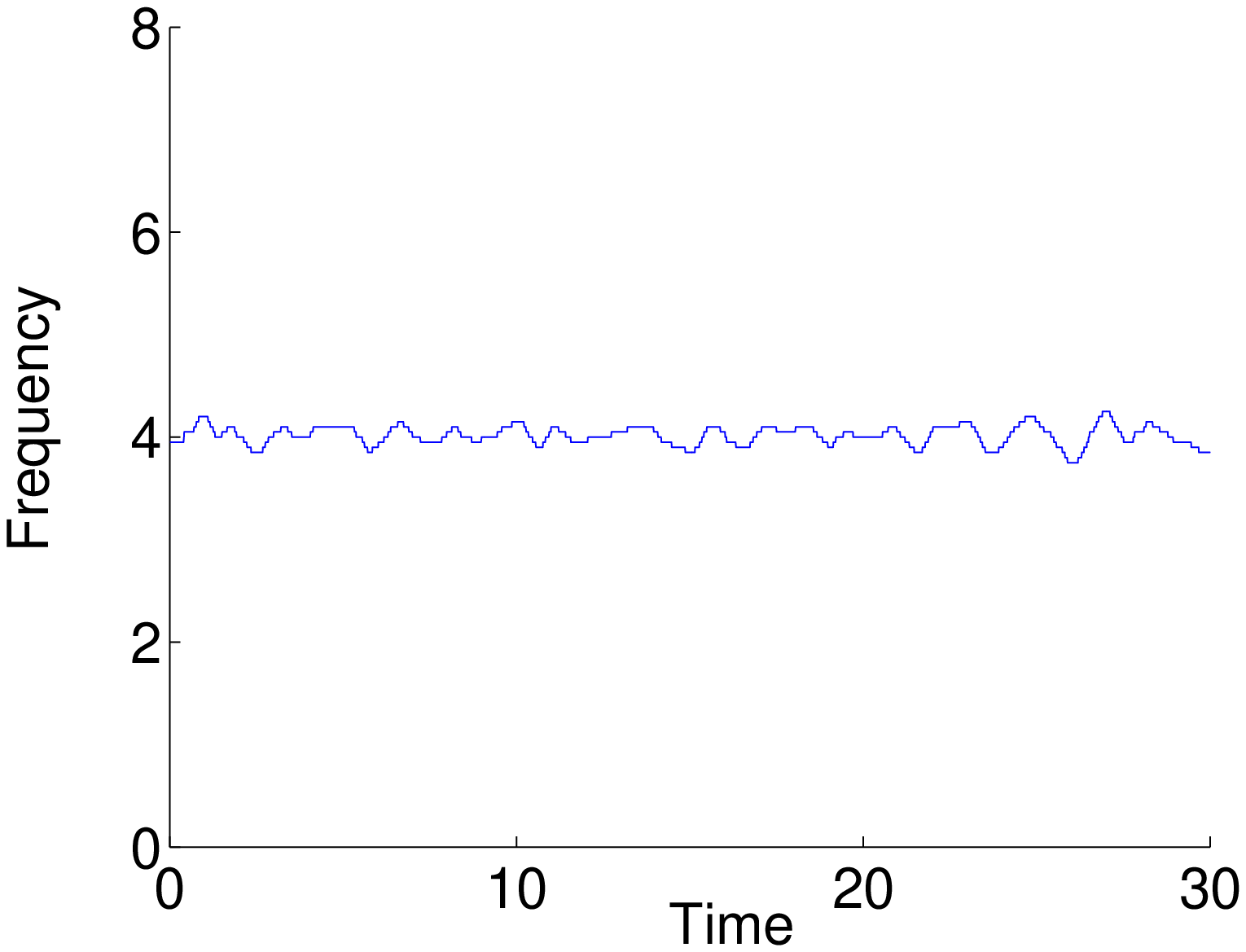}}\subfloat{\includegraphics[width=0.22\textwidth]{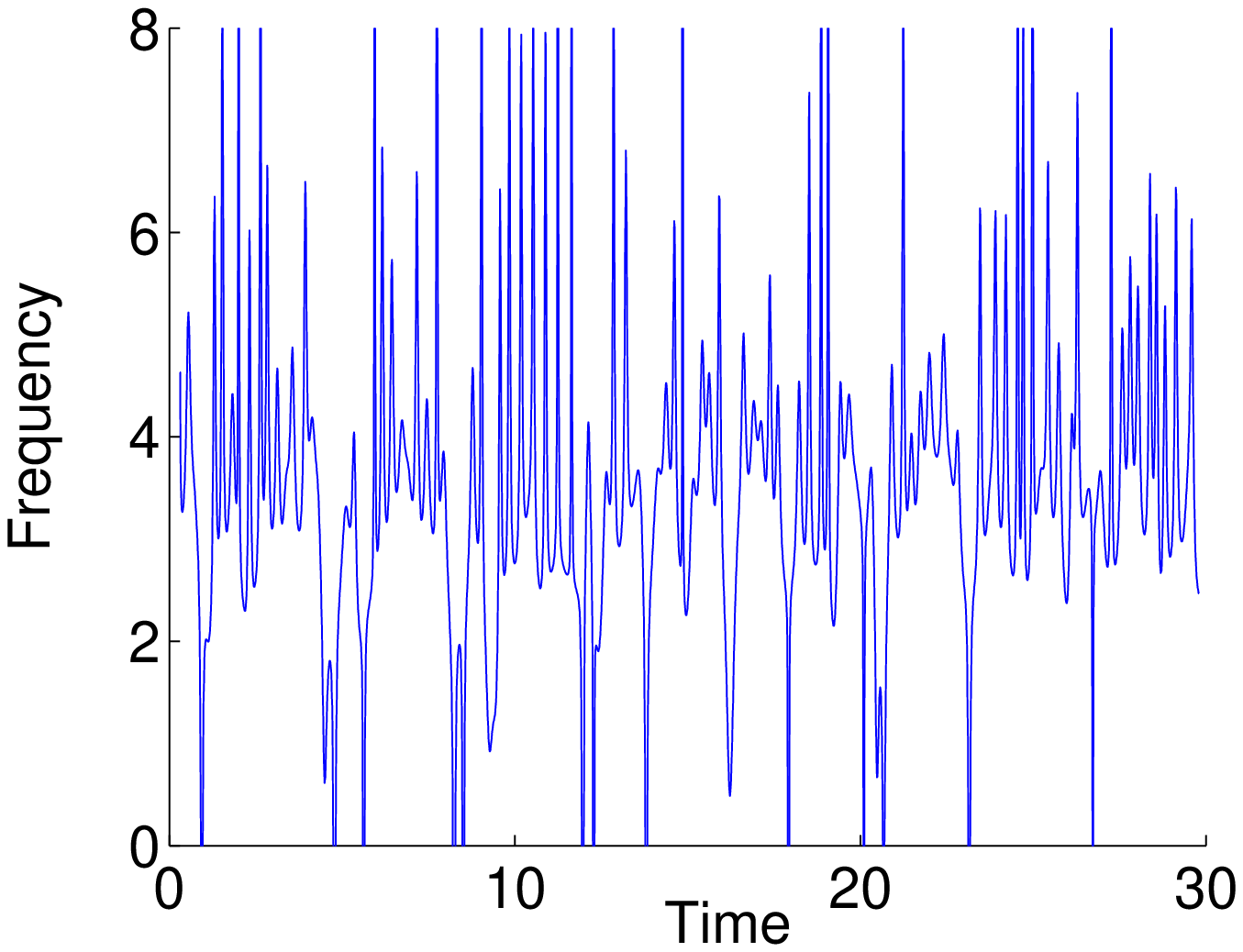}}\caption{\label{Fig4} The signal $f(t)=\cos(8\pi t)+N_{t}$, $t\in[0,30]$,
where $N_{t}$ is a realization of Gaussian white noise with mean
$0$ and variance $\sigma^{2}=0.4$. The samples are taken at $t_{n}=0.1n+T'a_{n}$,
$T'=0$ (left images) and $T'=0.08$ (right images). The computation
of $\mathrm{IF}_{S}f$ is fairly robust to the noise, while that of
$\mathrm{IF}_{H}f$ gives very poor results.}
\end{figure}

\begin{figure}[H]
\centering{}\subfloat{\includegraphics[width=0.27\textwidth]{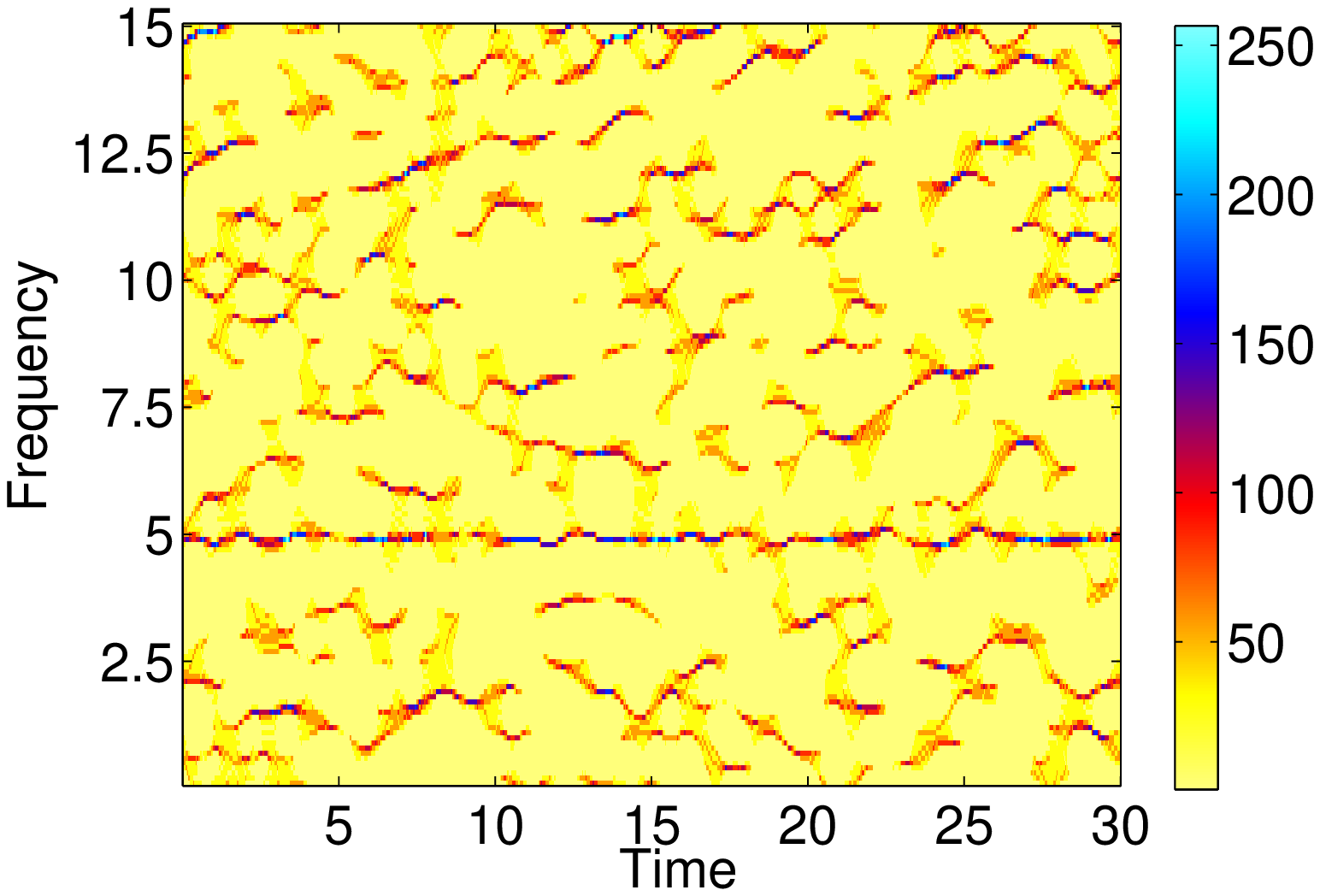}}\subfloat{\includegraphics[width=0.22\textwidth]{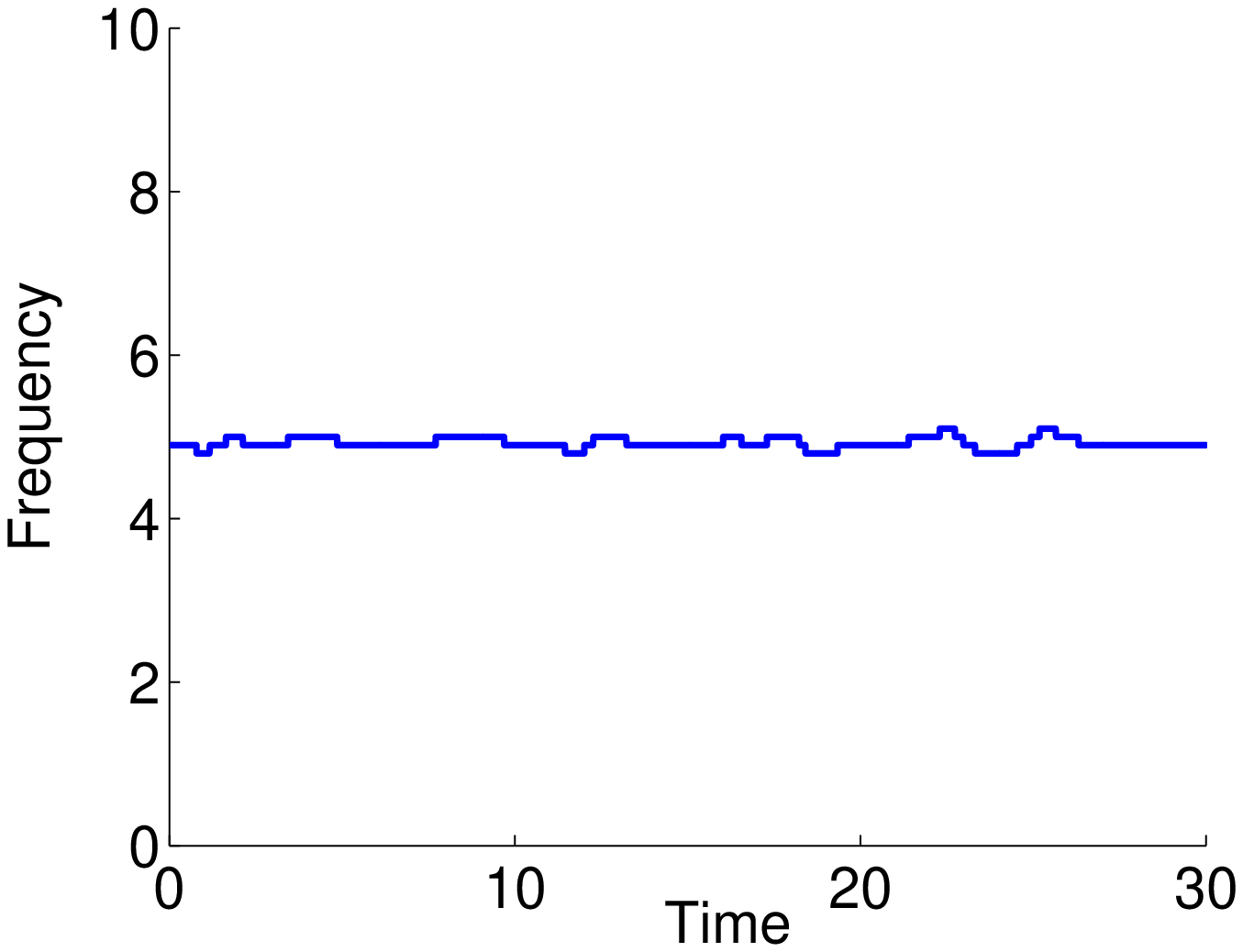}}\subfloat{\includegraphics[width=0.22\textwidth]{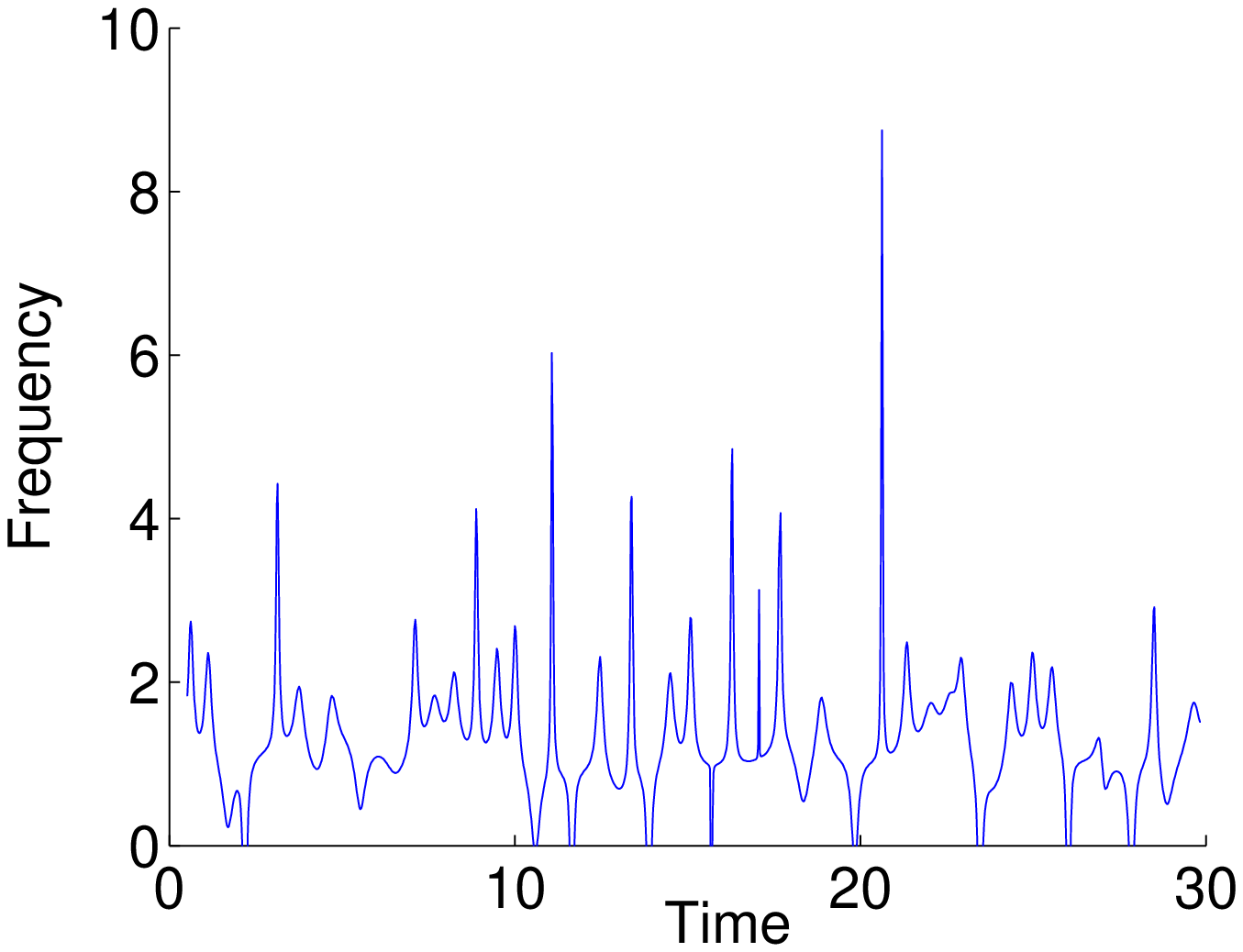}}\caption{\label{Fig5} The signal $f(t)=\cos(10\pi t)$, $t\in[0,30]$, with
samples taken at $t_{n}=0.25n+0.2a_{n}$. The IIF is the constant
$5$. Note that this signal is heavily undersampled, with a sampling
rate of less than half its Nyquist rate. We take $\gamma=6$ for the
STFT Synchrosqueezing computation, which produces a good result for
$\mathrm{IF}_{S}f$ despite the low sampling rate, but the bandlimited
reconstruction method cannot determine $\mathrm{IF}_{H}f$. In this
example, we also show the time-frequency plot of $S^{\alpha,\gamma}$
(the first image) to illustrate how the IF curve appears in it, before
we extract it out.}
\end{figure}

\begin{figure}[H]
\begin{centering}
\subfloat{\includegraphics[width=0.27\textwidth]{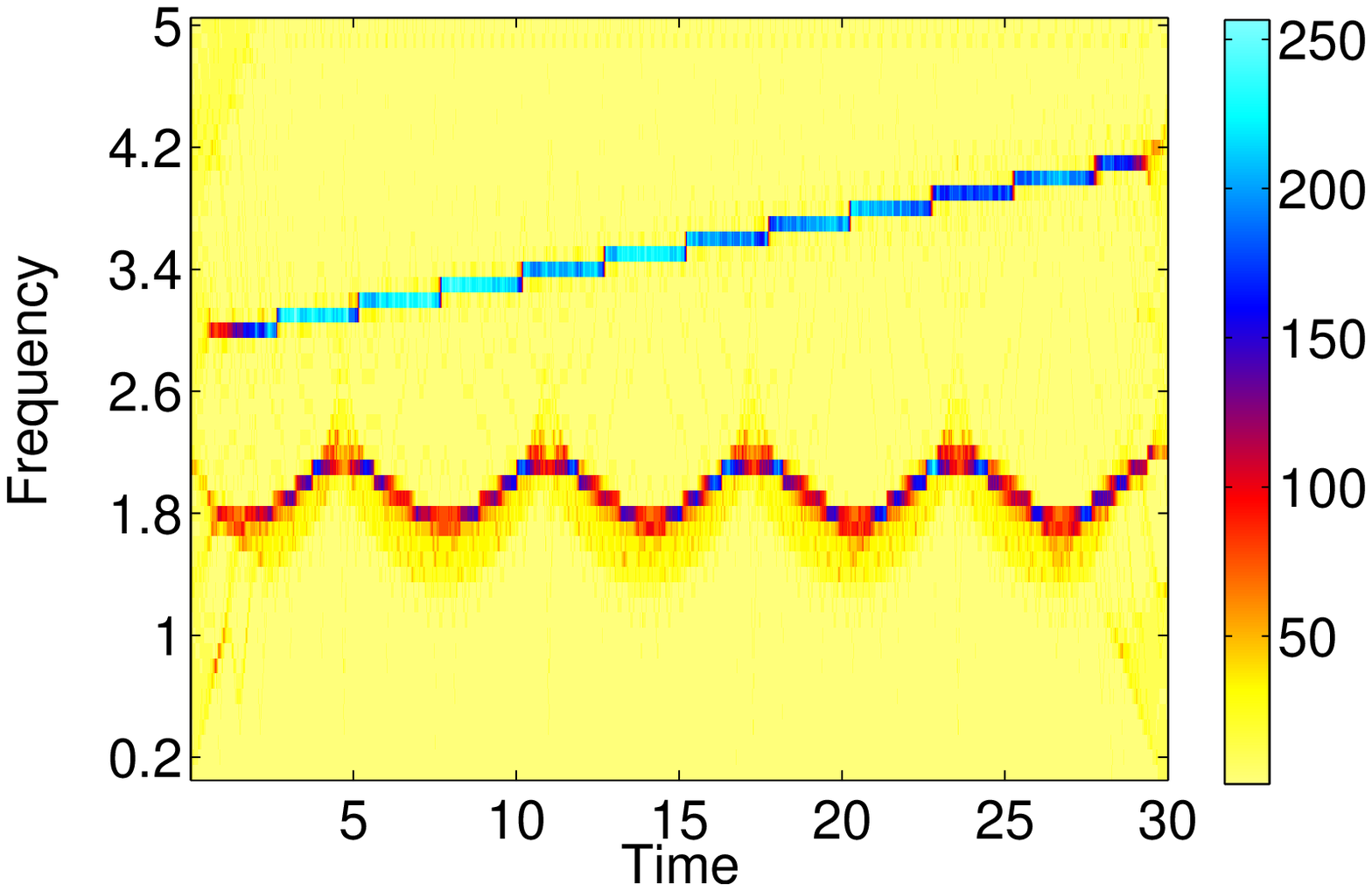}}\subfloat{\includegraphics[width=0.22\textwidth]{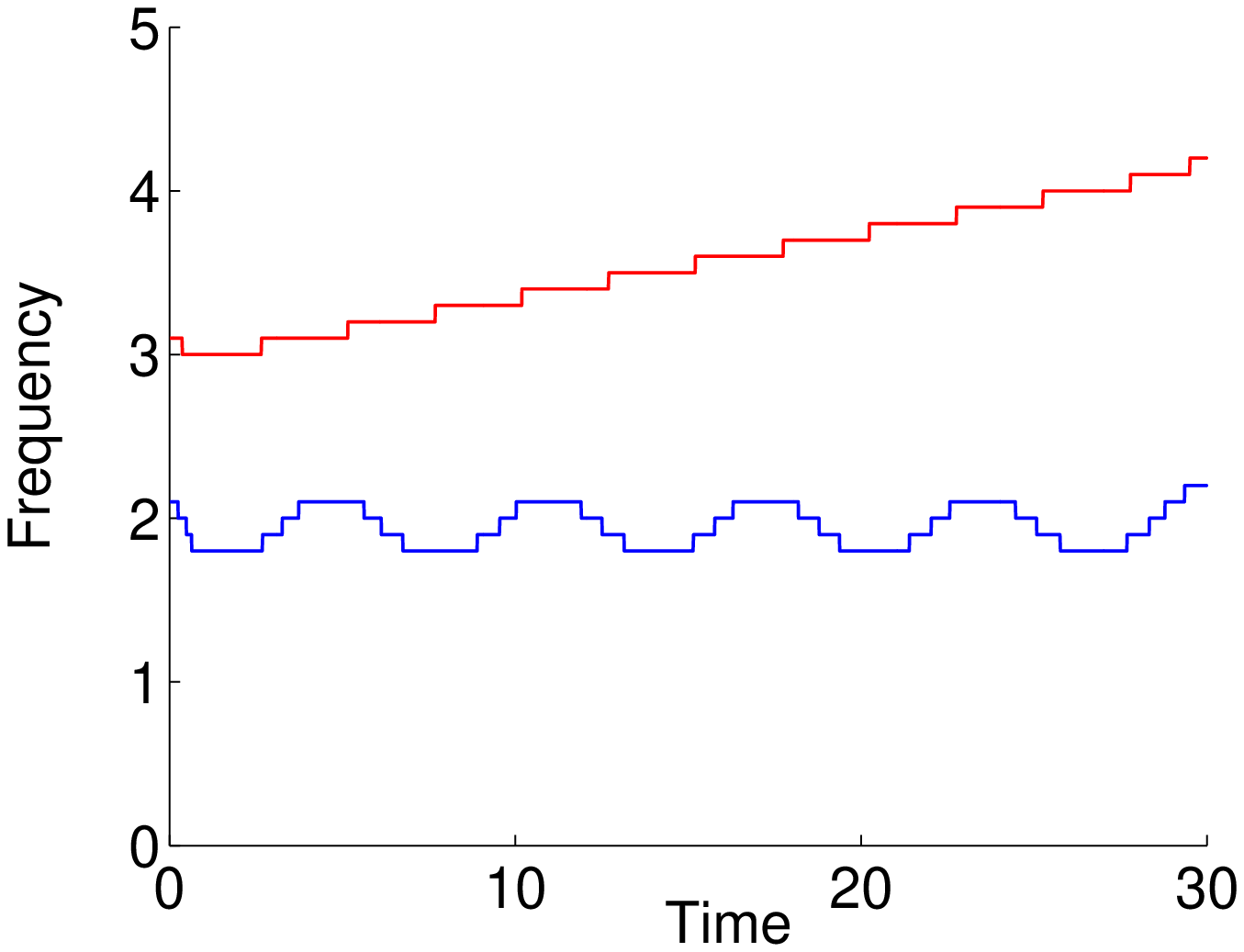}}\subfloat{\includegraphics[width=0.22\textwidth]{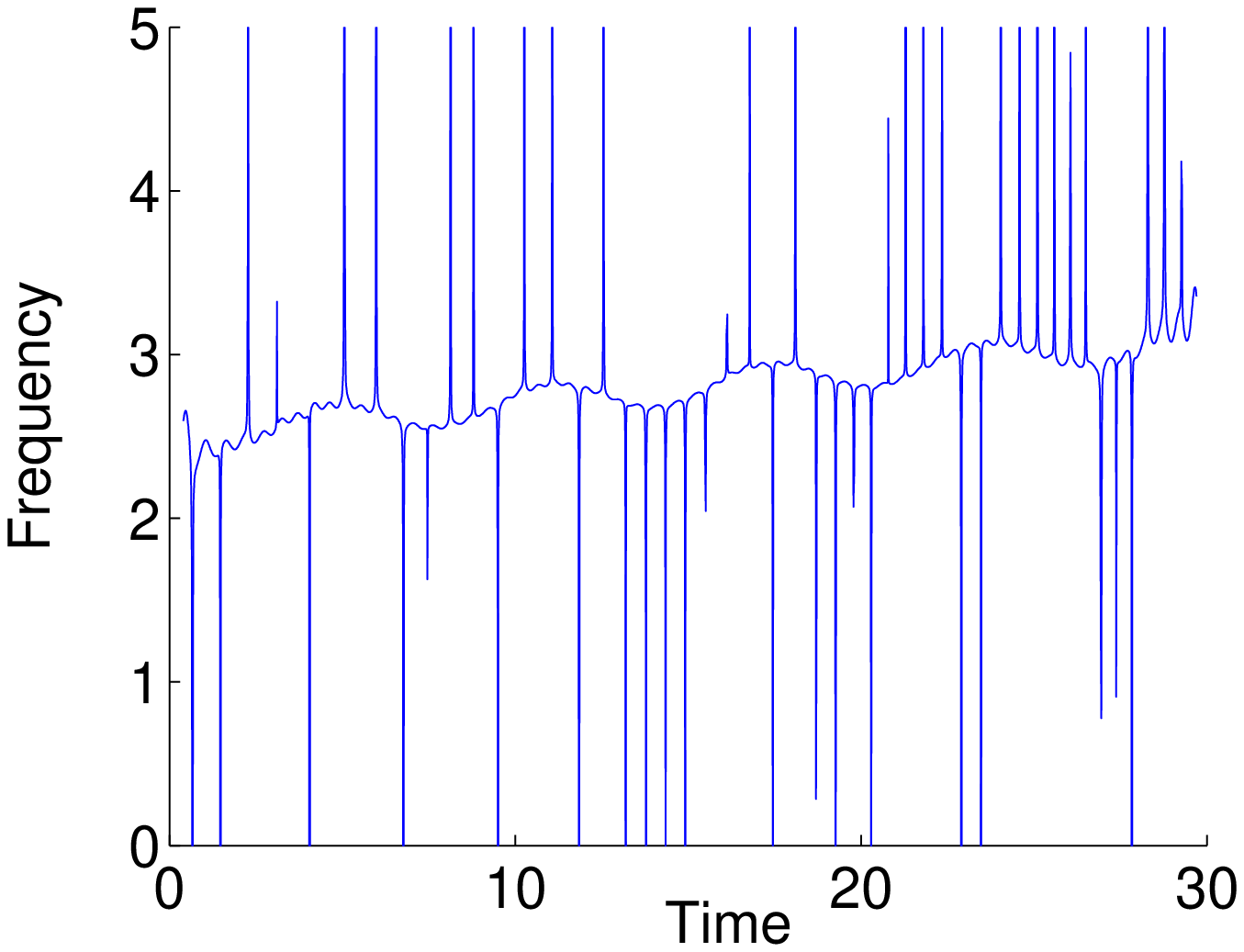}} 
\par\end{centering}

\centering{}\subfloat{\includegraphics[width=0.27\textwidth]{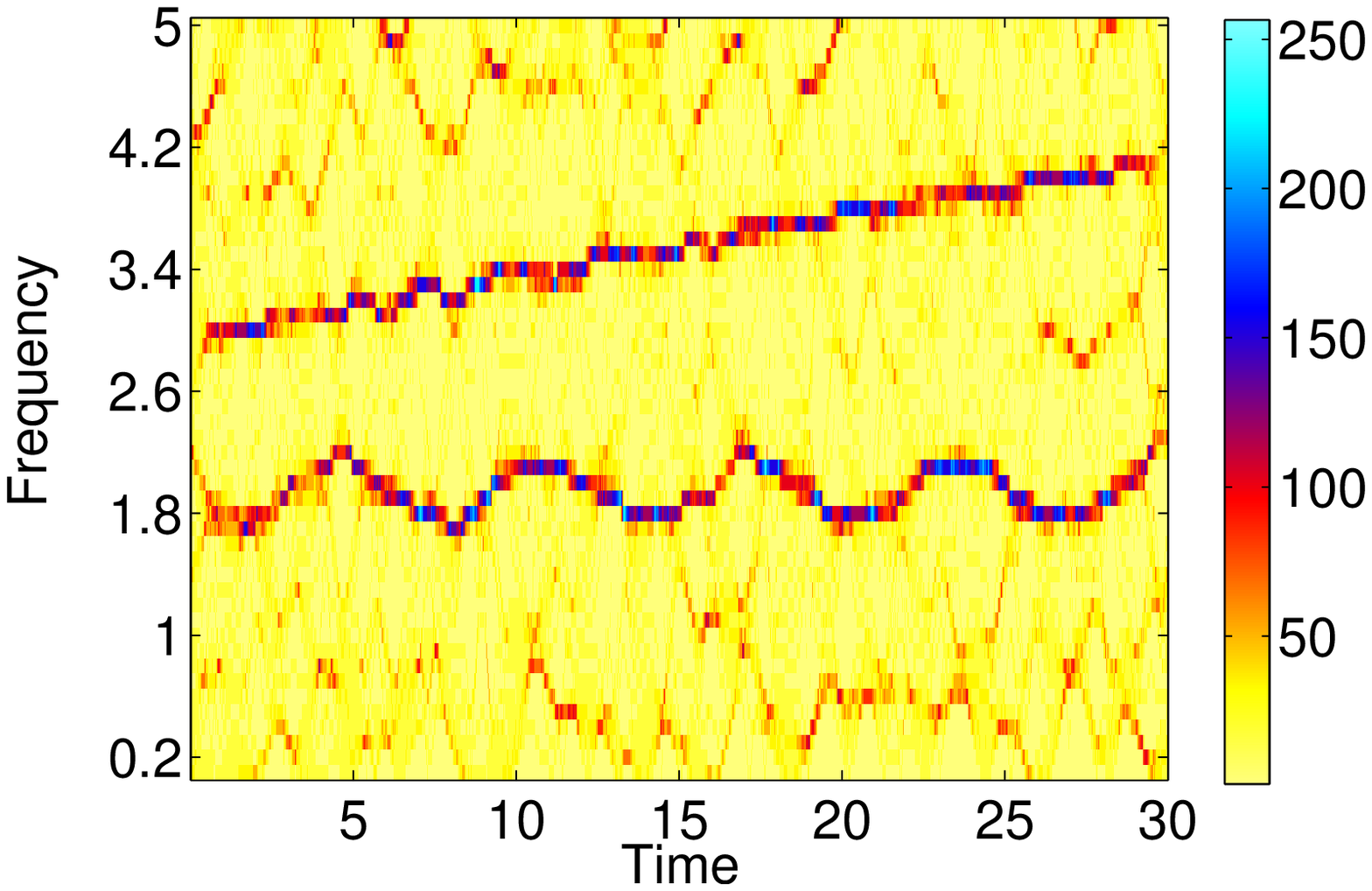}}\subfloat{\includegraphics[width=0.22\textwidth]{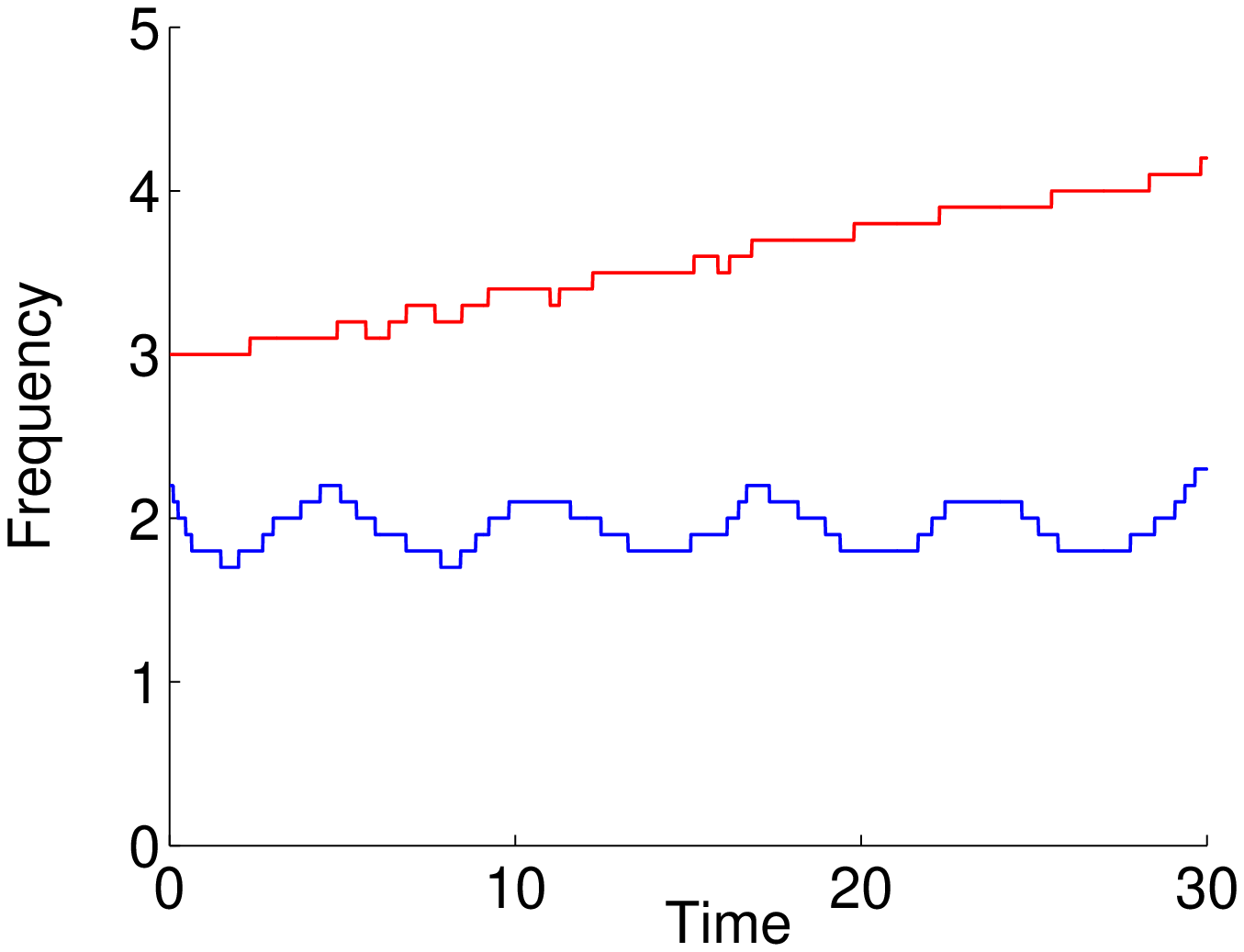}}\subfloat{\includegraphics[width=0.22\textwidth]{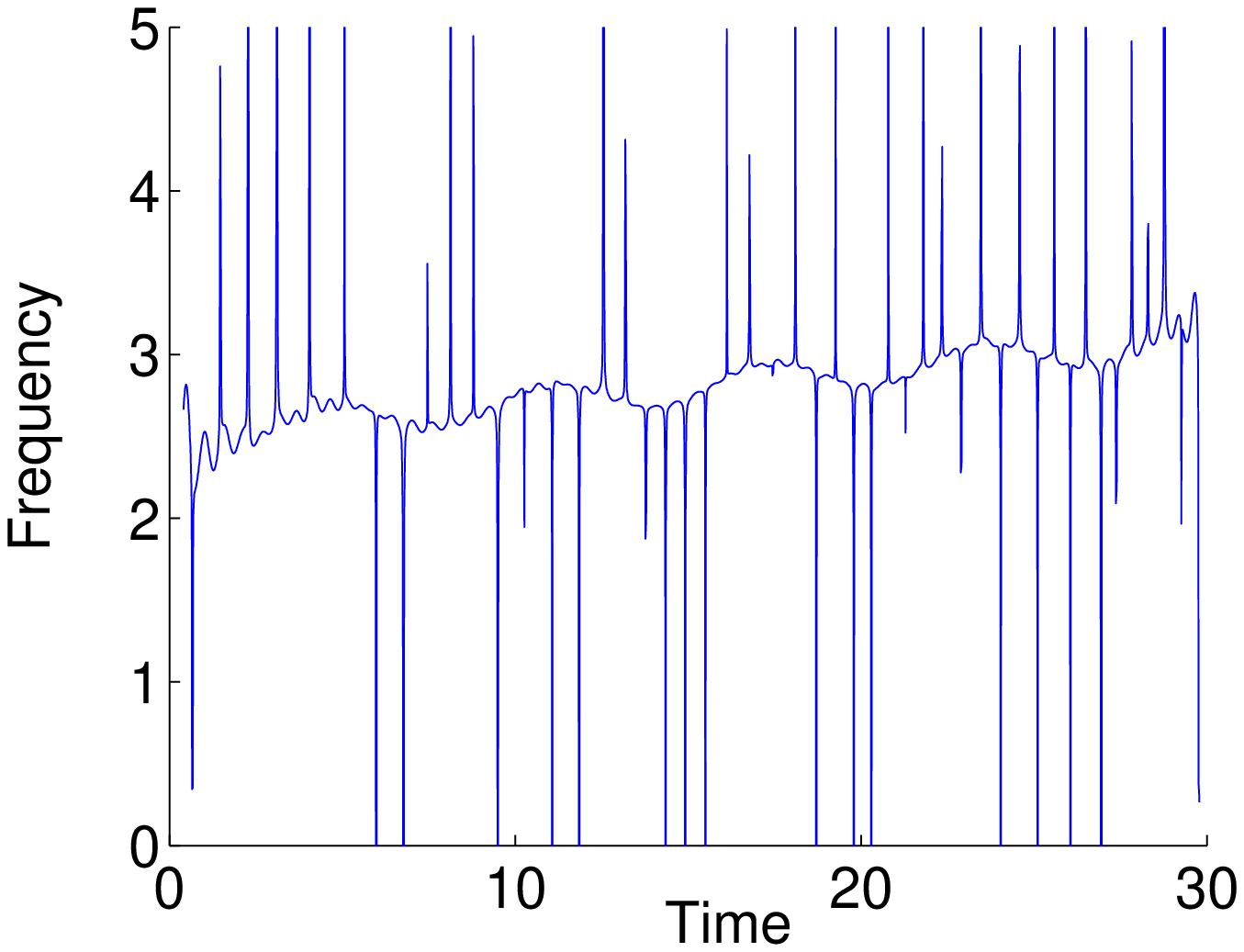}}\caption{\label{Fig6} The two-component signal $f(t)=\cos(2\pi(2t+0.2\cos t))+\cos(2\pi(3t+0.02t^{2}))$,
$t\in[0,30]$, with samples taken at $t_{n}=0.1n+T'a_{n}$, $T'=0$
(top images) and $T'=0.08$ (bottom images). We would like to recover
both elements of the IIF set $\{2-0.2\sin t,\,3+0.04t\}$, and the
$\mathrm{IF}_{S}$ calculation succeeds in doing this. In contrast,
the $\mathrm{IF}_{H}$ concept cannot separate the components, instead
roughly giving their average $2.5+0.02t-0.1\sin t$, and its computation
also exhibits spurious singularities.}
\end{figure}

\begin{figure}[H]
\begin{centering}
\subfloat{\includegraphics[width=0.25\textwidth]{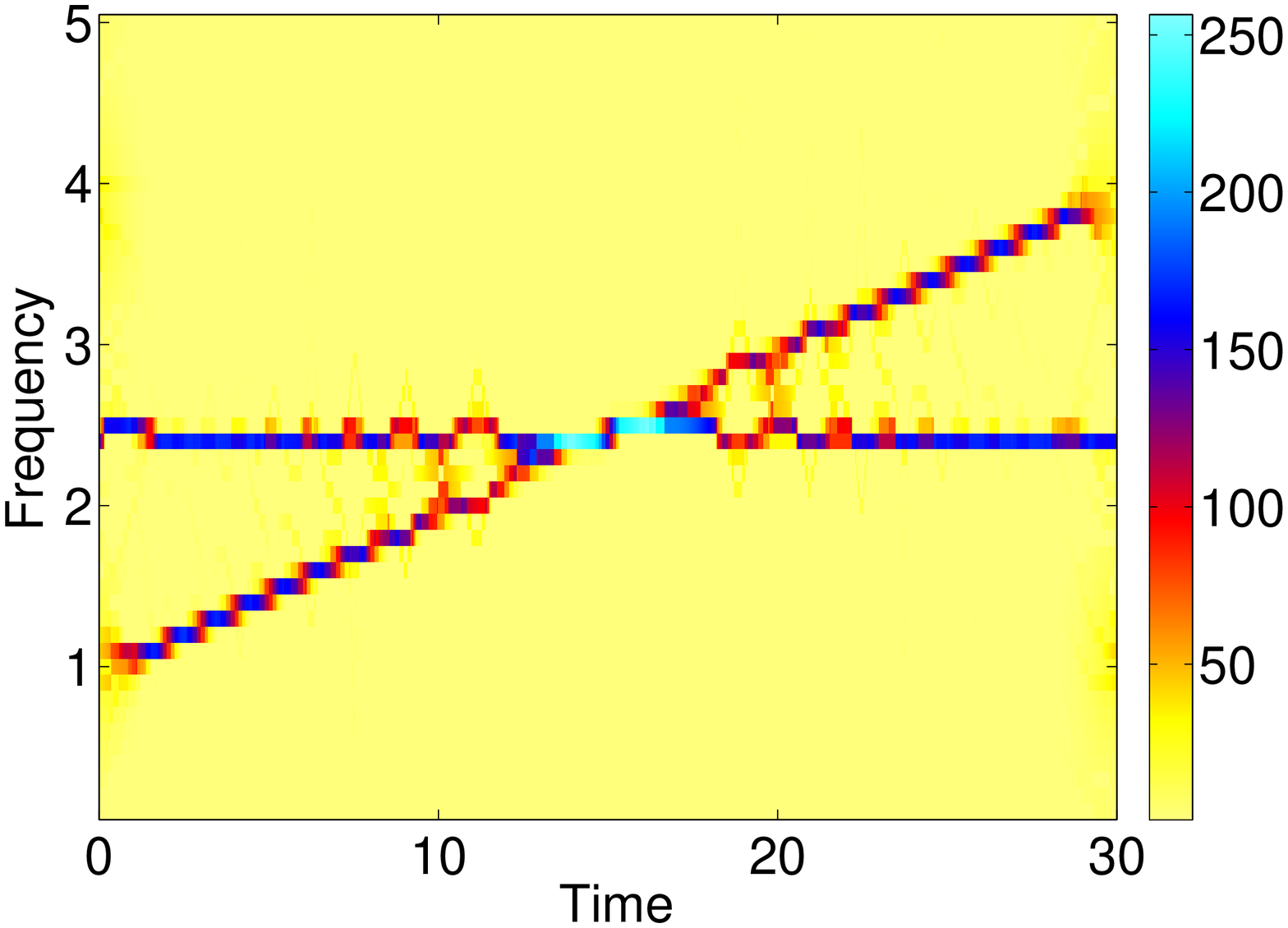}}\subfloat{\includegraphics[width=0.22\textwidth]{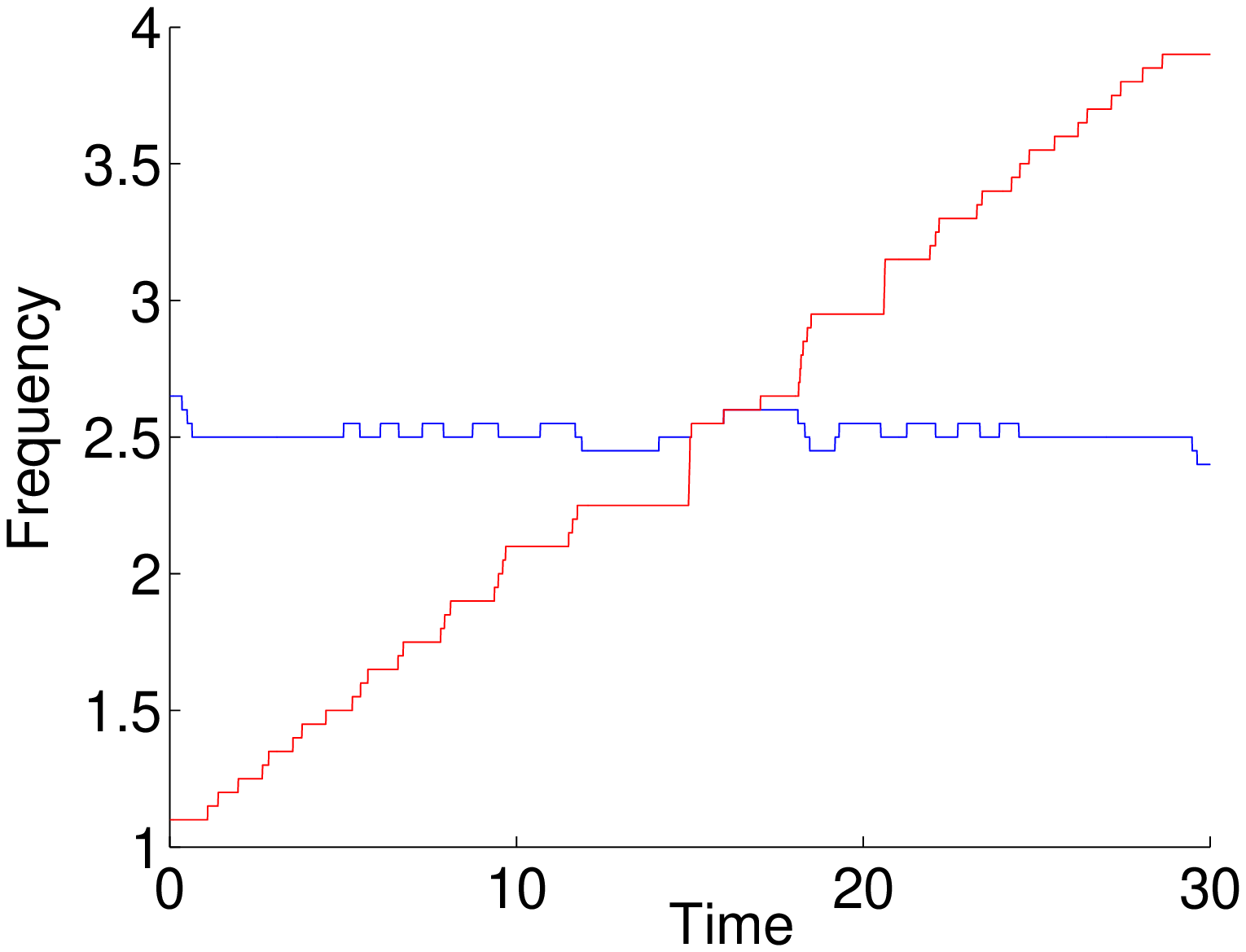}}\subfloat{\includegraphics[width=0.22\textwidth]{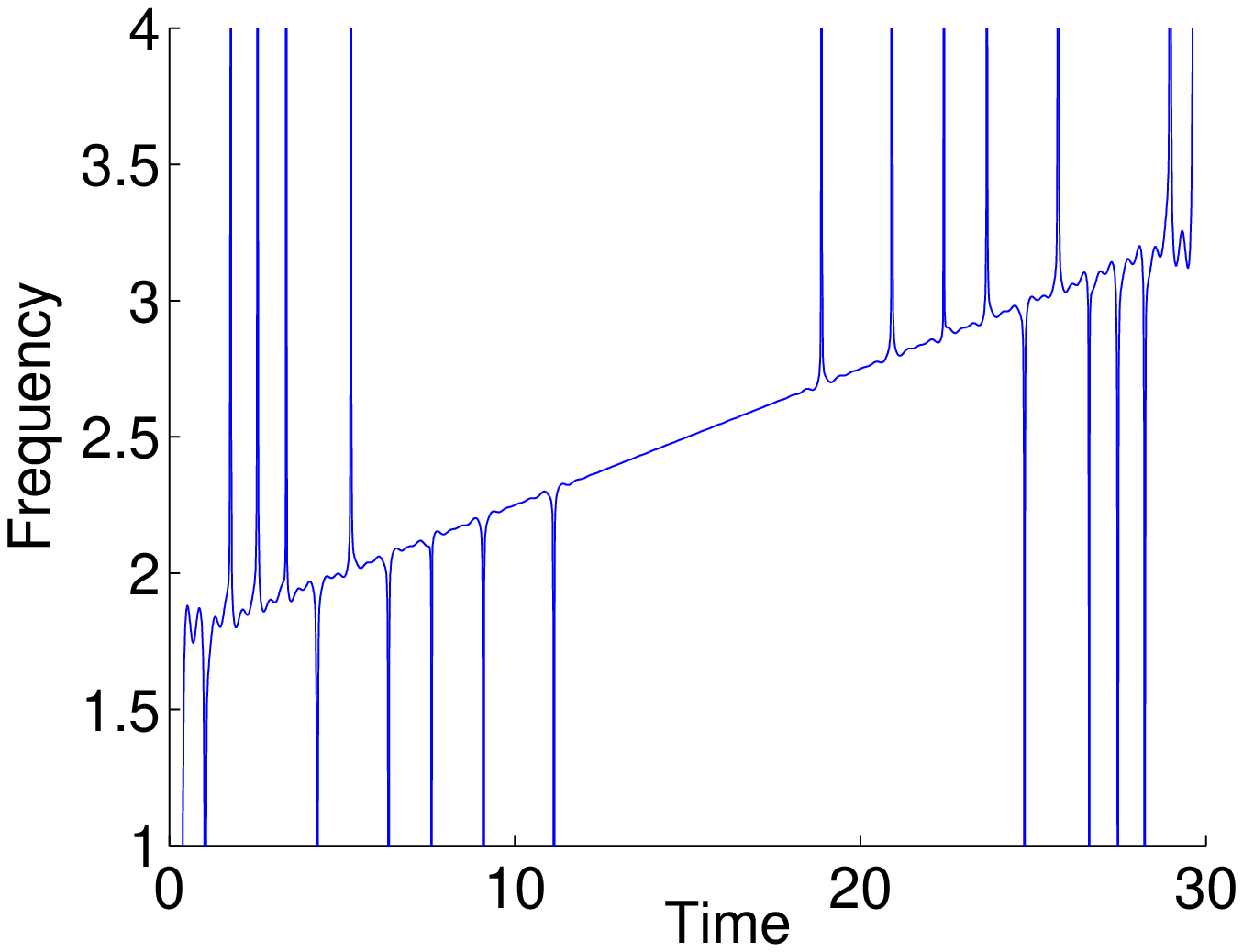}} 
\par\end{centering}

\centering{}\subfloat{\includegraphics[width=0.25\textwidth]{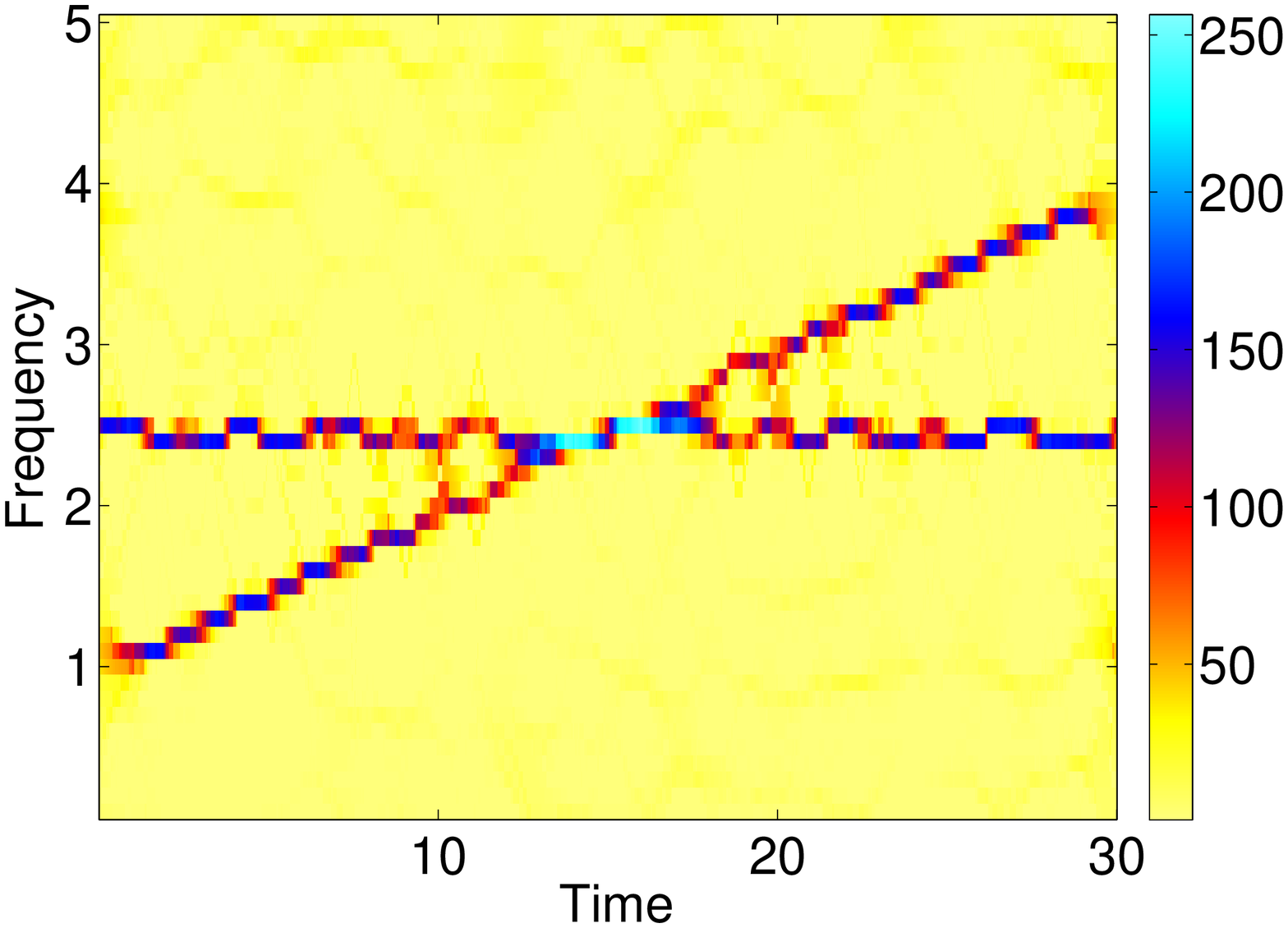}}\subfloat{\includegraphics[width=0.22\textwidth]{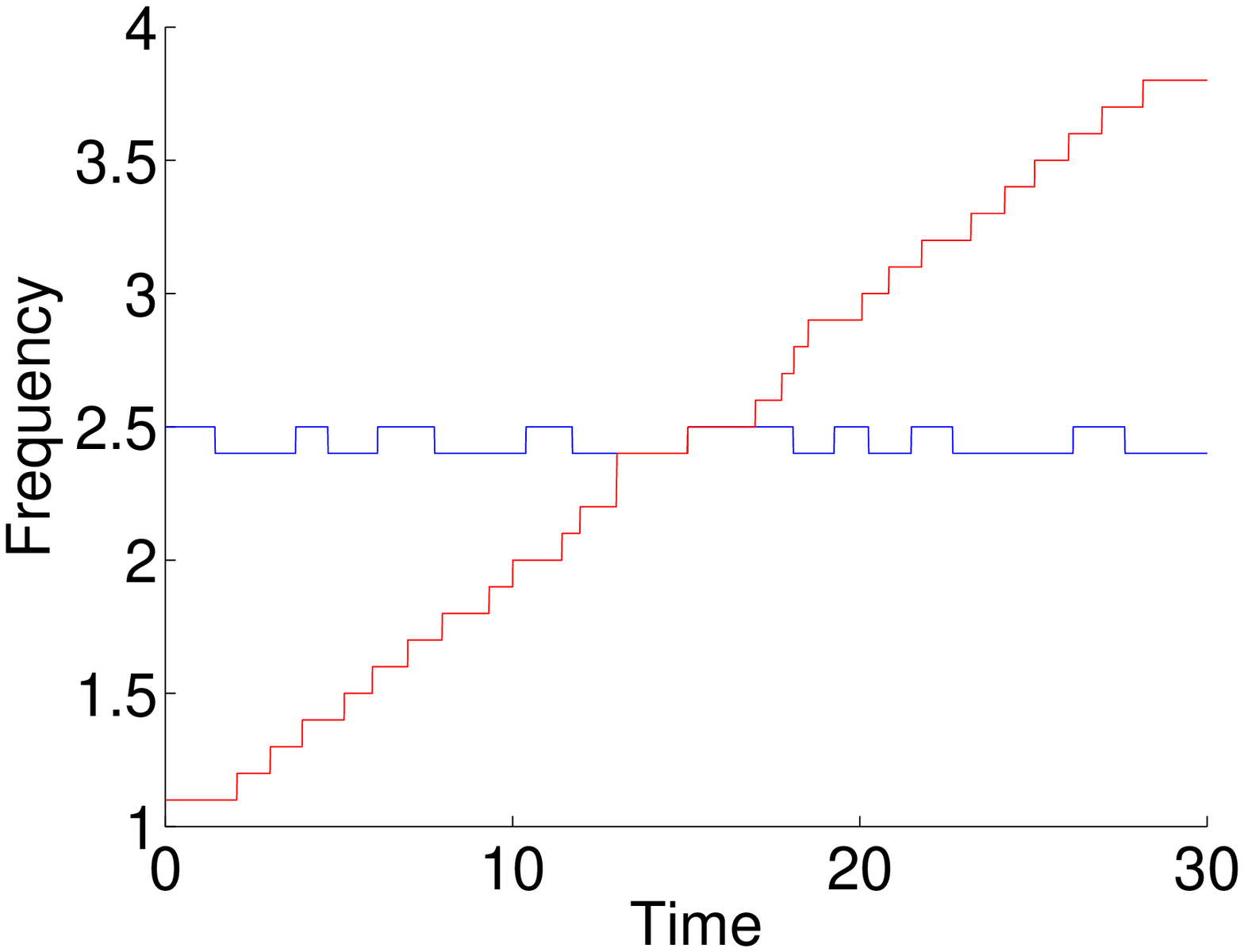}}\subfloat{\includegraphics[width=0.22\textwidth]{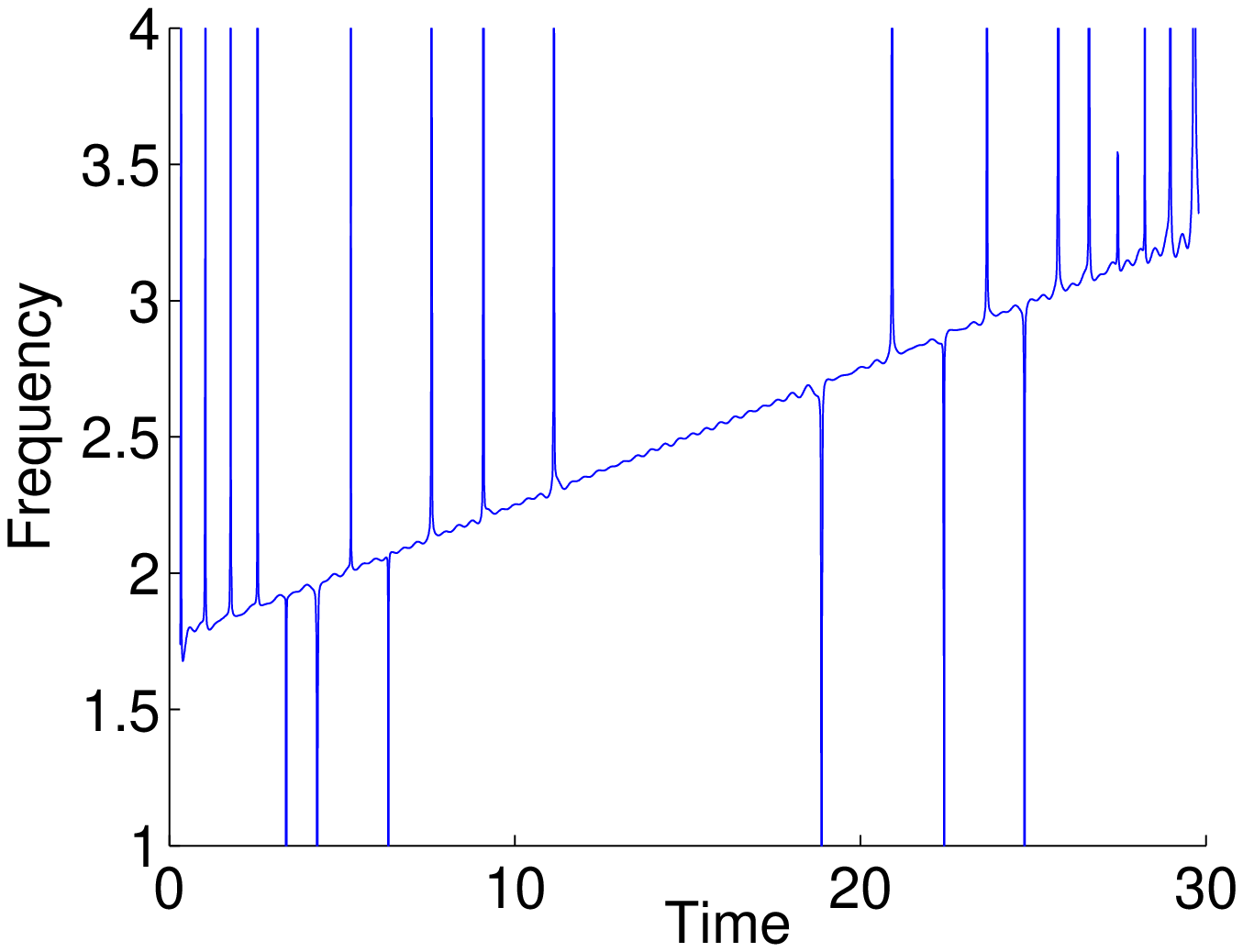}}\caption{\label{Fig7} The signal $f(t)=\cos(5\pi t)+\cos(2\pi(t+0.05t^{2}))$,
$t\in[0,30]$, with samples taken at $t_{n}=0.1n+T'a_{n}$, $T'=0$
(top images) and $T'=0.08$ (bottom images). This example is similar
to the previous one but the IIF curves cross each other. However,
the plot of $S^{\alpha,\gamma}$ shows a sharp separation between
the two curves around the crossover point and can distinguish between
them clearly.}
\end{figure}

The examples in Figures \ref{Fig3}-\ref{Fig7} show the advantages
of the $\mathrm{IF}_{S}$ approach over the traditional $\mathrm{IF}_{H}$
concept. In Figure \ref{Fig5}, the bandlimited reconstruction method
fails to recover $f$ accurately due to the low sampling rate, and
this is reflected in the computation of $\mathrm{IF}_{H}f$, but STFT
Synchrosqueezing still manages a good result. This can be possibly
explained by the fact that the Nyquist rate is essentially a concept
for bandlimited signals and is only relevant in that context, whereas
our STFT Synchrosqueezing theory is built around $\mathcal{B}_{\epsilon,d}$
signals and only estimates their IIF, not the signals themselves.
In Figures \ref{Fig4}, \ref{Fig6} and \ref{Fig7}, our sampling
rate is high enough and the bandlimited reconstruction method can
accurately determine $f$ itself, but the subsequent calculation of
$\mathrm{IF}_{H}f$ amplifies the effects of any noise or numerical
roundoff errors. In contrast, we find that $\mathrm{IF}_{S}f$ is
robust to such disturbances. We also note that determining a meaningful
IF for the type of signal in Figure \ref{Fig7} is often difficult
\cite{DEGKTT92} and such signals are certainly not in the class $\mathcal{B}_{\epsilon,d}$,
but the result is nevertheless very good.\\

We now discuss a real-world problem in electrocardiography (ECG) to
which our methods are applicable. In addition to describing the heart's
electrical activity, the ECG signal contains information about a signal
describing respiration. It is important in many clinical situations
to be able to determine properties of this respiration signal from
the ECG signal. For example, in an examination for tachycardia during
sleeping, where only the ECG signal and no respiration signal is recorded,
it allows for the detection and classification of sleep apnea. It
is well known in the ECG field that the customary surface ECG signal
is influenced by respiration, since inhalation and exhalation change
the thoracic electrical impedance, which suggests that the respiration
signal can be estimated from the ECG signal. The ECG-Derived Respiration
(EDR) class of techniques \cite{MMZM85}, in development since the
late 1980s, accomplish this and have proved to be a useful clinical
tool. We now show that STFT Synchrosqueezing provides an alternative
method to extract key information about the respiration signal from
an ECG signal. We can find the instantaneous frequency profile of
the respiration signal, which gives a more precise and adaptive description
of respiration than many of the existing techniques.\\

In Figure \ref{FigResp1}, we are given the lead II ECG signal and
the true respiration signal of a healthy $30$ year old male, recorded
over an $8$ minute interval. The sampling rates of the ECG and respiration
signals are respectively 512Hz and 64Hz. We take the R peaks (the
sharp, tall spikes in the first image in Figure \ref{FigResp1}) of
the ECG signal and use these samples to approximate the IF of the
respiration signal, without using any knowledge of the actual respiration
signal. We do not have samples of the respiration signal itself, but
we can view the R peaks as samples of an envelope of the ECG signal,
and based on the physiological facts discussed above, this envelope
would be expected to have the same IF profile as the actual respiration
signal. We apply the methods from Sections \ref{SecMain} and \ref{SecBL}
to compute $\mathrm{IF}_{S}$ and $\mathrm{IF}_{H}$ from the R peaks,
and use the actual, recorded respiration signal to compare the validity
of our results. Let the ECG and true respiration signals be denoted
by $E(t)$ and $R(t)$ respectively, where $t\in[0,480]$ seconds.
There are $589$ R peaks appearing at times $t_{k}\in[0,480]$, $t_{k}<t_{k+1}$,
$1\leq k\leq589$. For the calculation of $\mathrm{IF}_{S}$, we use
the impulse train-like function 
\[
\tilde{f}_{Rpeaks}(t)=\left\{ \begin{array}{ll}
(t_{k}-t_{k-1})E(t) & \mbox{ if }t=t_{k}\\
0 & \mbox{ otherwise }
\end{array}\right.,
\]

\noindent and for $\mathrm{IF}_{H}$, we simply use the samples $\{E(t_{k})\}$.
The results are shown in Figure \ref{FigResp1} below. The third image
in Figure \ref{FigResp1} shows that the $\mathrm{IF}_{S}$ computed
from $\tilde{f}_{Rpeaks}$ is a good approximation to the $\mathrm{IF}_{S}$
of the true respiration signal $R(t)$. On the other hand, the $\mathrm{IF}_{H}$
determined from the R peaks in the fourth image has little in common
with $\mathrm{IF}_{H}R(t)$. In fact, $\mathrm{IF}_{H}R(t)$ is often
negative and admits no obvious interpretation, unlike the $\mathrm{IF}_{S}$.
This is likely a result of the fact that ECG measurements usually
contain large amounts of noise, which $\mathrm{IF}_{H}$ does not
handle well, and many standard noise reduction techniques are not
applicable here since they would smooth out the R peaks. In Figure
\ref{FigResp2}, it can be seen that the spacing of respiration cycles
in $R(t)$ is reflected by $\mathrm{IF}_{S}$ of the R peaks; closer
spacing corresponds to higher $\mathrm{IF}_{S}$ values, and wider
spacing to lower $\mathrm{IF}_{S}$ values.

\begin{figure}[H]
\centering{}\subfloat{\includegraphics[width=0.22\textwidth]{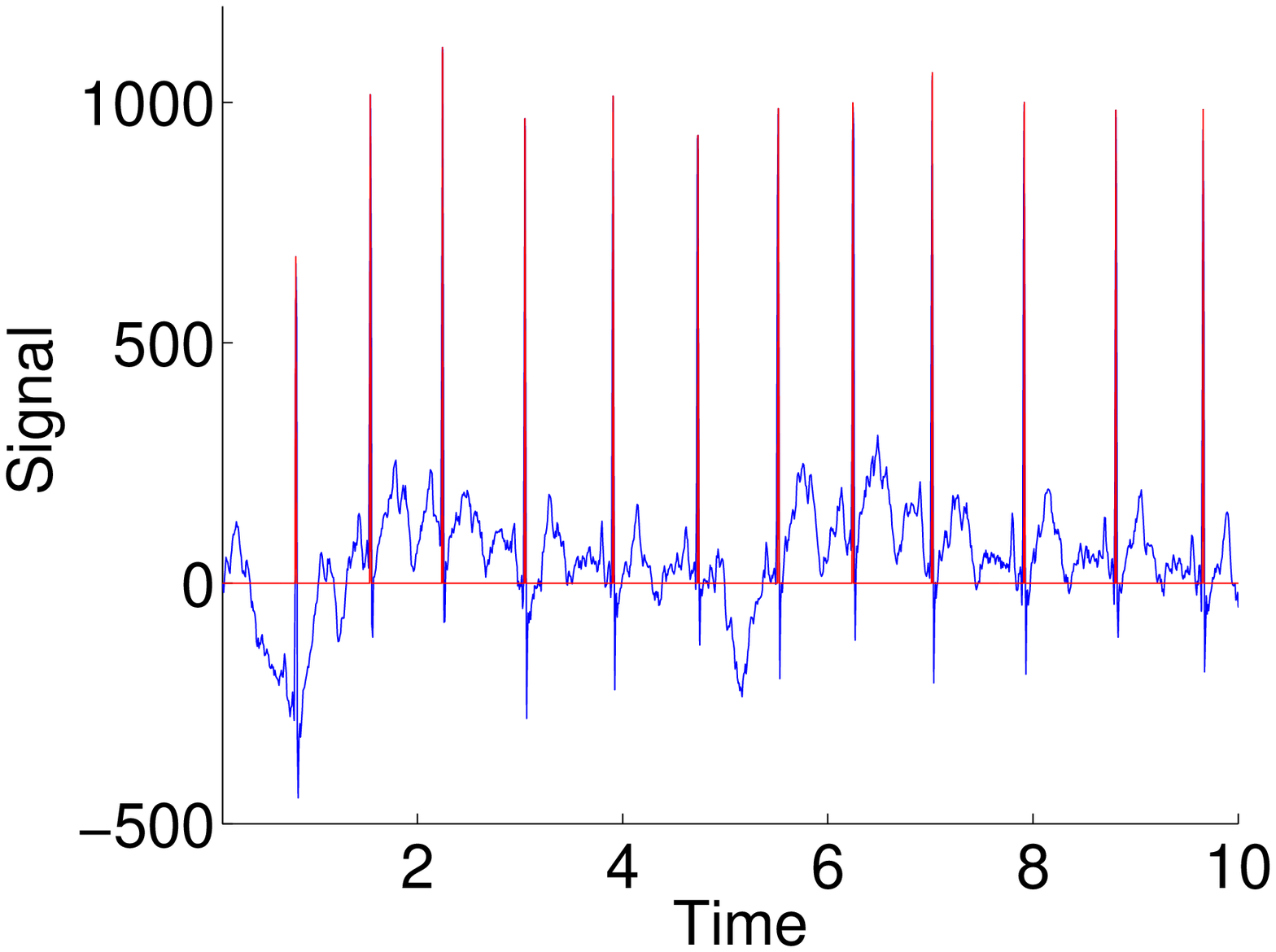}}\subfloat{
\includegraphics[width=0.22\textwidth]{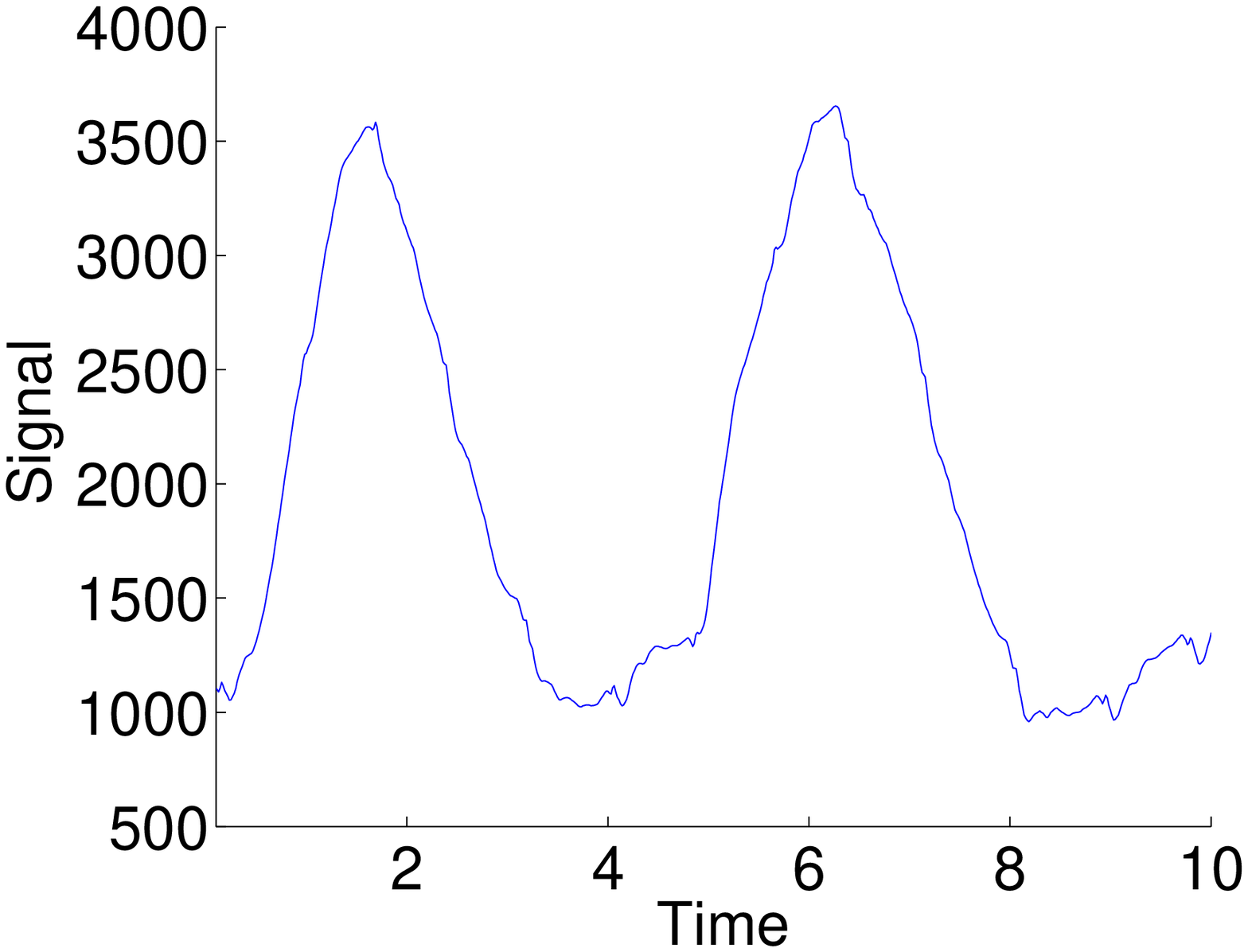}} \subfloat{
\includegraphics[width=0.22\textwidth]{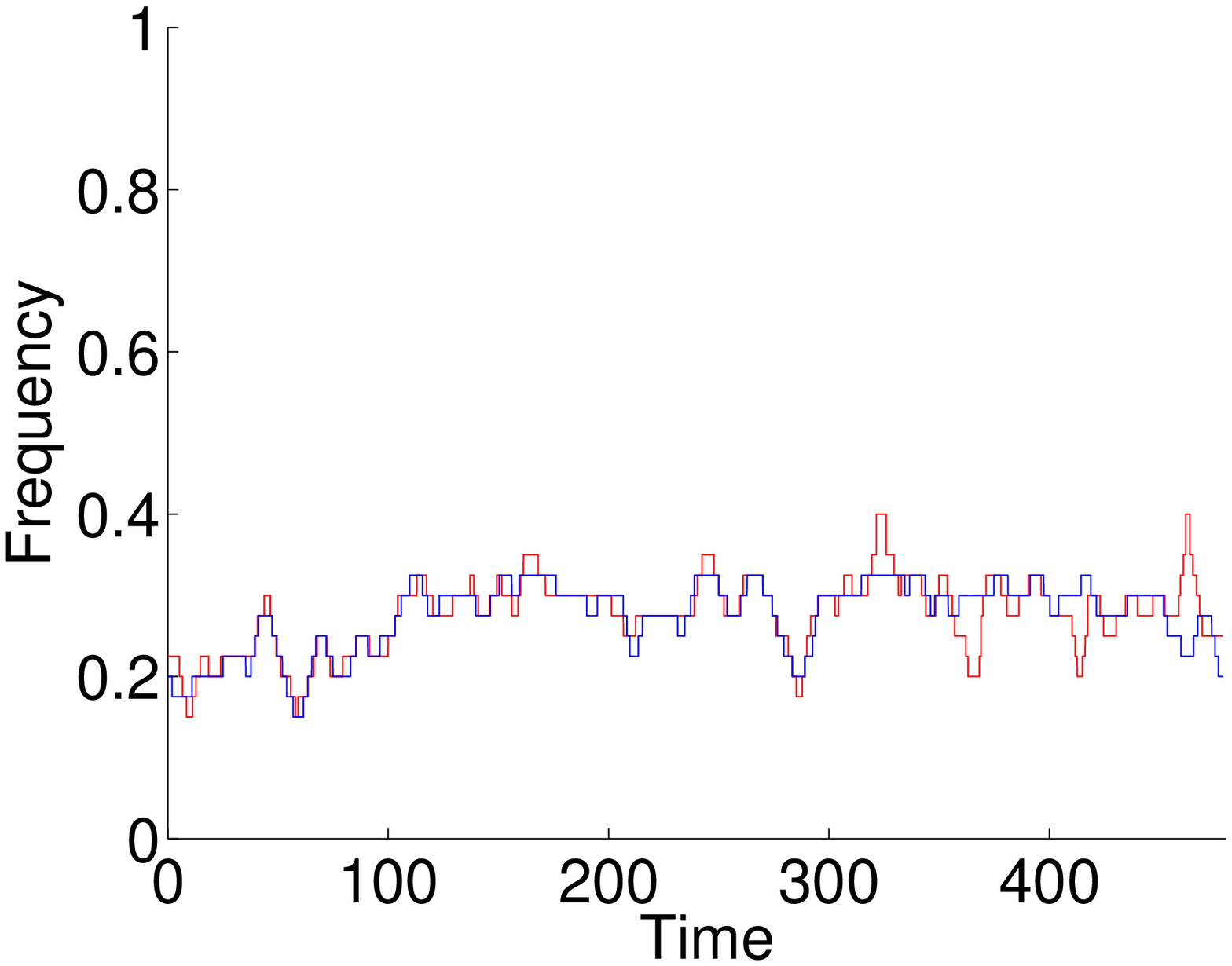}}\subfloat{
\includegraphics[width=0.22\textwidth]{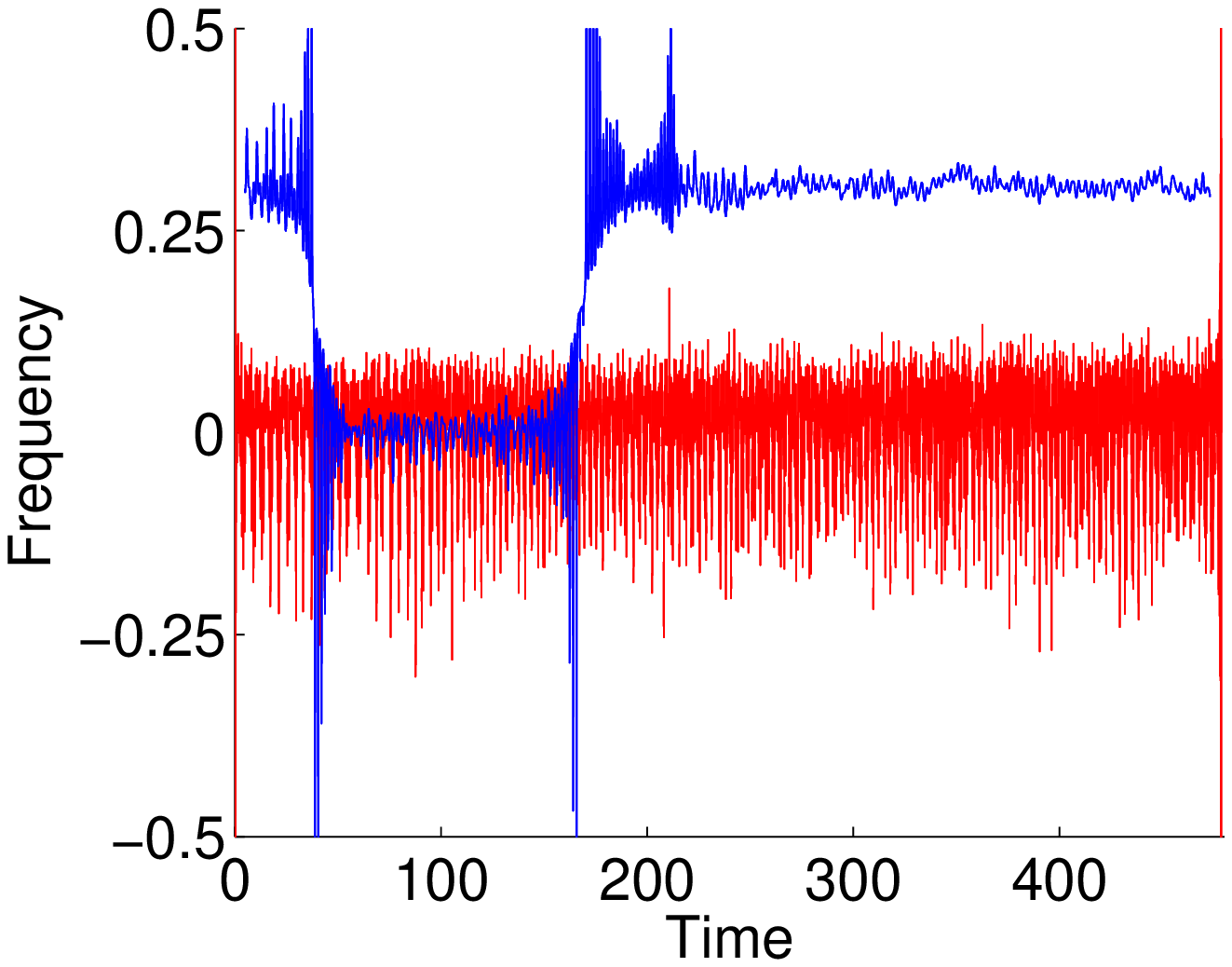}} \caption{\label{FigResp1} First Image: A 10-second part of the ECG signal.
The R peaks are highlighted in red. Second Image: A 10-second part
of the respiration signal. Third Image: The $\mathrm{IF}_{S}$ computed
from $\tilde{f}_{Rpeaks}$ (blue) and the $\mathrm{IF}_{S}$ of the
actual respiration signal $R(t)$ (red). Fourth Image: The $\mathrm{IF}_{H}$
computed from $E(t_{k})$ using bandlimited reconstruction (blue)
and the $\mathrm{IF}_{H}$ of $R(t)$ (red).}
\end{figure}

\begin{figure}[H]
\centering{}\subfloat{\includegraphics[width=0.95\textwidth]{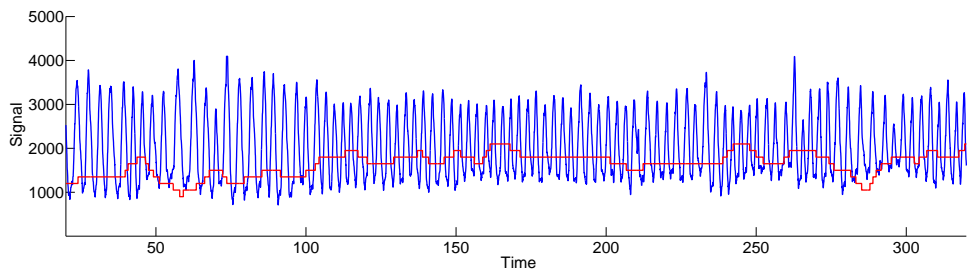}}
\caption{\label{FigResp2} The first $300$ seconds of $R(t)$ (blue) with
the $\mathrm{IF}_{S}$ estimated from $\tilde{f}_{Rpeaks}$ (red)
superimposed on top of the graph of $R(t)$. }
\end{figure}

\begin{acknowledgement*}
The authors would like to thank Professor Ingrid Daubechies for many
valuable discussions in the course of this work, and Dr. Ray F. Lee
for assistance in collecting the respiration signal. H.-T. Wu also
acknowledges discussions with Dr. Shu-Shya Hseu and Prof. Chung-Kang
Peng. The authors acknowledge support by FHWA grant DTFH61-08-C-00028.
\end{acknowledgement*}
\bibliographystyle{siam}
\bibliography{nonuniform}

\end{document}